\documentclass[11pt,oneside,reqno]{amsart}

\usepackage[margin=1in]{geometry}
\usepackage{amsmath,amssymb,amsthm,textcomp}
\usepackage{amsfonts,graphicx}
\usepackage{amsaddr}
\usepackage[mathscr]{eucal}
\usepackage{pgfplots}
\pgfplotsset{compat=1.15}
\usepackage{subfig,tikz}
\usepackage{array} 
\usepackage{booktabs}
\usepackage{multicol}
\usepackage{multirow}
\usepackage[utf8]{inputenc}
\usepackage{bm}

\usepackage{color}
 
\usepackage{float}
\usepackage{hyperref}
 \hypersetup{colorlinks,citecolor=blue}
\hypersetup{colorlinks=true,linkcolor=red,filecolor=magenta,    urlcolor=blue}

\usepackage{pslatex}
 \usepackage{array}
\usepackage[square,numbers]{natbib}
\allowdisplaybreaks[1]
\theoremstyle{definition}

\newcommand{\ncom}{\newcommand}

\ncom{\integ}[4]{\int_{#1}^{#2}\,{#3}\,d{#4}}
\ncom{\vspan}[1]{{{\rm\,span}\{ #1 \}}}
\ncom{\dm}[1]{ {\displaystyle{#1} } }
\ncom{\ri}[1]{{#1} \index{#1}}

\newtheorem{theorem}{\bf Theorem}[section]
\newtheorem{remark}{\bf Remark}[section]
\newtheorem{proposition}{Proposition}[section]
\newtheorem{lemma}{Lemma}[section]
\newtheorem{corollary}{Corollary}[section]

\newtheoremstyle
{remarkstyle}
{}
{11pt}
{}
{}
{\bfseries}
{:}
{     }
{\thmname{#1} \thmnumber{#2} }

\theoremstyle{remarkstyle}

\def\t{\tau}

\DeclareFontFamily{U}{BOONDOX-calo}{\skewchar\font=45 }
\DeclareFontShape{U}{BOONDOX-calo}{m}{n}{
  <-> s*[1.05] BOONDOX-r-calo}{}
\DeclareFontShape{U}{BOONDOX-calo}{b}{n}{
  <-> s*[1.05] BOONDOX-b-calo}{}
\DeclareMathAlphabet{\mathcalboondox}{U}{BOONDOX-calo}{m}{n}
\SetMathAlphabet{\mathcalboondox}{bold}{U}{BOONDOX-calo}{b}{n}
\DeclareMathAlphabet{\mathbcalboondox}{U}{BOONDOX-calo}{b}{n}

\begin{document}
\color{black}       
\title{Martingale Characterizations of Non-Homogeneous Counting Processes and Their Fractional Variants}



\author{Kartik Tathe$^1$ and Sayan Ghosh$^2$}
\address{Department of Mathematics, Birla Institute of Technology and Science, Pilani, Hyderabad Campus, Jawahar Nagar, Kapra Mandal, Medchal District, Telangana 500078, India. }
\thanks{$^2$Corresponding author}
\email{$^1$kartikvtathe@gmail.com,$^2$sayan@hyderabad.bits-pilani.ac.in}





\begin{abstract}
This paper investigates the martingale characterizations of non-homogeneous counting processes and their fractional generalizations. We show that a weighted sum of non-homogeneous Poisson processes (NPPs) is the non-homogeneous generalized counting process (NGCP). Both the compensated and exponential forms of martingale characterization for NGCP are obtained, and are shown to be equivalent. Moreover, we provide martingale characterizations for various time-changed variants of the NGCP and their Skellam versions using stable and/or inverse stable subordinators.
\end{abstract}

\maketitle


\section{Introduction}
Counting processes play a fundamental role in stochastic modeling and have wide-ranging applications in queueing theory, reliability analysis, finance, and insurance mathematics. Among these, the Poisson process stands out as the most prominent and well-studied example. A non-negative integer-valued stochastic process $\{N(t)\}_{t\geq 0}$ satisfying $N(t_1) \leq N(t_2)$ for $t_1 < t_2$ with independent and stationary increments is called a Poisson process with rate $\lambda>0$ if $N(t)$ has a Poisson distribution with parameter $\lambda t$. If we replace $\lambda$ with an appropriate function of time $\lambda(t)$, we obtain the non-homogeneous Poisson process (NPP). Some applications of NPP include modeling call arrivals in telecommunication networks, traffic flow and traffic accidents and software failures which may occur at different rates during various time periods. 

Fractional and generalized extensions of the Poisson process have been widely explored in the literature to describe systems exhibiting non-Markovian features and over-dispersion. \citet{LASKIN2003201} introduced the time fractional Poisson process (TFPP) by replacing the ordinary derivative in the Kolmogorov-Feller equations governing the process with a Riemann-Liouville fractional derivative, while \citet{Meerschaert2011} obtained the same process by time changing (replacing the time variable) the Poisson process with an independent inverse stable subordinator. This fractional modification allows for non-stationary and dependent increments. Later, \citet{Leonenko2017} introduced a non-homogeneous version of TFPP. Allowing flexibility of various jump sizes (multiple arrivals) with different intensities, \citet{Crescenzo2016} introduced the generalized fractional counting process (GFCP). The non-homogeneous extension called the NGFCP and its special case NGCP were subsequently investigated by \citet{Kataria2025}. \citet{Orsingher2012} introduced the space fractional Poisson process (SFPP) by introducing fractionality in the backward shift operator acting on the state space of the process. They further proposed the space-time fractional Poisson process (STFPP), which unifies the TFPP and the SFPP, by incorporating the Caputo fractional derivative into the governing equation of the probability generating function.

An important probabilistic approach to studying and characterizing counting processes is through their martingale properties. A martingale is a stochastic process for which the conditional expectation of a future state given the past states is equal to the present state. In other words, martingales can be used to model fair processes that evolve over time since the expected values of the states of the process at all time points are the same. \citet{Watanabe1964} provided a martingale characterization of the classical Poisson process based on the inter-arrival time approach. An equivalent formulation can also be obtained through stochastic calculus techniques (see, for example, \citet{Shreve2004}). For an extension of this characterization to the positive real half-line, we refer the reader to \citet{Bremaud1975}. \citet{Aletti2018} established a martingale characterization for the TFPP by time changing the Poisson process using the inverse stable subordinator. Recently, \citet{Dhillon2024} investigated martingale characterizations of the generalized counting process (GCP) and its time-changed variants, namely GFCP and the mixed fractional counting process (MFCP). In the present work, we obtain martingale characterizations for various non-homogeneous counting processes and their fractional generalizations with respect to space and/or time, thereby encompassing all such characterizations in the literature as special cases. Some major applications of martingales include analysis of betting games, financial markets, information theory, stochastic network calculus, queuing systems and branching processes. Thus establishing martingale characterizations enables one to use the flexible stochastic models discussed in this paper for numerous practical situations.  

The rest of the paper is organized as follows. In Section \ref{sec 2}, we present preliminary definitions and essential results relevant to our work. In Section \ref{sec 3}, we establish martingale characterizations of the NPP and the NGCP. We also demonstrate the equality between a weighted sum of NPPs and the NGCP along with the equivalence between the compensated and exponential forms of martingales associated with these processes. Section~\ref{sec 4} provides the martingale characterizations of various fractional extensions of the processes considered. Specifically, we examine the non-homogeneous time fractional, space fractional, space–time fractional, and mixed fractional variants of the NPP and NGCP. Furthermore, we derive the martingale characterizations of the Skellam-type processes corresponding to these fractional stochastic models.

\section{Preliminaries}\label{sec 2} In this section, we present some preliminary results and definitions which will be used later in the paper.
\subsection{Martingale}\label{def:martingale}
Let $\{X(t)\}_{t \ge 0}$ be a stochastic process adapted to a filtration $\{\mathcal{F}_t\}_{t \ge 0}$. 
The process $\{X(t)\}_{t\ge 0}$ is a martingale if
\begin{enumerate}
    \item $\mathbb{E}[|X(t)|] < \infty$ for $t \ge 0$;
    \item $\mathbb{E}[X(t) \mid \mathcal{F}_s] = X(s)$ for $0 \le s \le t$.
\end{enumerate}
\subsection{Non-homogeneous Poisson process}\label{def:nhpp}
Let $\{\mathcal{N}(t)\}_{t\geq 0}$ be a counting process with a deterministic and time-dependent intensity function 
$\lambda:[0,\infty)\to [0,\infty)$ and the corresponding cumulative rate function 
$\Lambda(t)=\int_{0}^{t}\lambda(u)\,du$ such that $\mathcal{N}(0)=0$. 
Then $\{\mathcal{N}(t)\}_{t\geq 0}$ is said to be a non-homogeneous Poisson process (NPP) if it has independent increments and its transition probabilities are given by
\begin{align*}
P(\mathcal{N}(t+h)=n\mid \mathcal{N}(t)=m) &=
\begin{cases}
1 - \lambda(t)h + o(h), & n=m, \\
\lambda(t)h + o(h), & n=m+1, \\
o(h), & n>m+1,
\end{cases}
\end{align*}
    where $o(h)\rightarrow 0$ as $h\rightarrow 0$.    
Equivalently, for $0 \le s \le t$, 
\begin{equation*}
\mathcal{N}(t+s) - \mathcal{N}(s) \sim \mathrm{Poisson}\!\left( \int_{s}^{t+s} \lambda(u)\,du \right).
\end{equation*}

\subsection{Non-homogeneous generalized counting process}\label{ngcp}
 Let $\{\mathcal{M}(t)\}_{t\geq 0}$ be a counting process with deterministic and time-dependent intensity functions $\lambda_j:[0,\infty)\to [0,\infty)$ for $ 1\leq j \leq k$ and cumulative rate functions $\Lambda_{j}(t)=\int_{0}^{t}\lambda_{j}(u)du$ such that $\mathcal{M}(0)=0$. Then $\{\mathcal{M}(t)\}_{t\geq 0}$ is a non-homogeneous generalized counting process (NGCP) if it has independent increments and its transition probabilities are given by (see \citet{Kataria2025})
\begin{align*}
P(\mathcal{M}(t+h)=n\mid\mathcal{M}(t)=m) &=
\begin{cases}
1 - \sum_{j=1}^{k}\lambda_j(t)h + o(h), & n=m, \\
\lambda_j(t)h + o(h), & n=m+j, j=1,2,\ldots,k, \\
o(h), & n>m+k,
\end{cases}
\end{align*}
where $o(h)\rightarrow 0$ as $h\rightarrow 0$.

\subsection{Stable subordinator and its inverse}\label{stable subordinator} A stable subordinator $\{D_{\alpha}(t)\}_{t\geq 0}$, $0<\alpha<1$, is a non-decreasing L\'{e}vy process. Its Laplace transform is given by $E(e^{-sD_\alpha(t)})=e^{-ts^\alpha}$, $s>0$ and the associated Bernstein function is $f(s)=s^\alpha$. Its first passage time $\{Y_\alpha(t)\}_{t\geq 0}$ is called the inverse stable subordinator and defined as  
\begin{equation*}
Y_\alpha(t):=\inf\{x\geq 0:D_{\alpha}(x)>t \}.
\end{equation*}
The mean and variance (see \citet{Leonenko2014}) of $Y_\alpha(t)$ are given by
\begin{align*}
\mathbb{E}(Y_{\alpha}(t))=\frac{t^\alpha}{\Gamma(\alpha+1)}, \quad
\mathbb{V}(Y_\alpha(t))=\left(\frac{2}{\Gamma(2\alpha+1)}-\frac{1}{\Gamma^2(\alpha+1)}\right)t^{2\alpha}.
\end{align*}

\subsection{Tempered stable subordinator}\label{TSS}
A tempered stable subordinator (TSS) $\{D_{\beta,\theta}(t)\}_{t\ge 0}$ with stability index $0<\beta<1$ and tempering paramter $\theta>0$ is obtained by exponentially tempering the distribution of $\beta$-stable subordinator $\{D_\beta(t)\}_{t\ge 0}$ (see \citet{Rosinski2007}). 
The mean and variance (see \citet{Gupta2020a}) of $D_{\beta,\theta}(t)$ are given by
\begin{equation*}
    \mathbb{E}\left[D_{\beta,\theta}(t)\right] = \beta\theta^{\beta-1}t,\quad\mathbb{V}\left[D_{\beta,\theta}(t)\right]=\beta(1-\beta)\theta^{\beta-2}t.
\end{equation*}

\subsection{Mixed stable subordinator and its inverse}\label{MSS} 
The mixed stable subordinator $\{L_{\alpha_1,\alpha_2}(t)\}_{t\ge0}$ (see \citet{Aletti2018}) is a non-decreasing Lévy process characterized by the following Laplace transform:
\begin{equation*}
    \mathbb{E}\left(e^{-sL_{\alpha_1,\alpha_2}(t)}\right)
    = e^{-t\left(C_1 s^{\alpha_1} + C_2 s^{\alpha_2}\right)}, \qquad s \ge 0,
\end{equation*}
where $C_1\ge 0,~C_2\ge0$, $C_1+C_2=1$ and the stability indices satisfy $\alpha_1<\alpha_2$. 
The inverse mixed stable subordinator $\{Y_{\alpha_1,\alpha_2}(t)\}_{t\ge0}$, 
also known as the first-passage time of $\{L_{\alpha_1,\alpha_2}(t)\}_{t\ge 0}$, 
is defined by
\begin{equation*}
    Y_{\alpha_1,\alpha_2}(t)
    := \inf\{s\ge0 : L_{\alpha_1,\alpha_2}(s) > t\}, \qquad t\ge0.
\end{equation*}
For $C_2=0$, $\{Y_{\alpha_1,\alpha_2}(t)\}_{t\ge 0}$ reduces to the inverse $\alpha_1$-stable subordinator $\{Y_{\alpha_1}(t)\}_{t\ge0}$ and a similar reduction holds for $C_1=0$.

\section{Martingale characterizations of non-homogeneous counting processes}\label{sec 3}
In this section, we obtain the martingale characterizations of NPP and NGCP. We show that the compensated and exponential martingale forms are equivalent for these processes. Also we provide a unique representation of the NGCP in terms of a weighted sum of the NPPs. 
\subsection{Martingale characterization of NPP}
The following result will be used to establish the martingale characterization of a NPP. 
\begin{lemma}\label{martingale lemma}
  Let $\{\mathcal{N}(t)\}_{t\ge 0}$ be a counting process with jump size $+1$ with rate $\lambda(t)$ such that $\mathcal{N}(0)=0$ and $\mathcal{F}_t=\sigma(\{\mathcal{N}(s)\},\,s\le t)$, $t\ge 0$ be a filtration generated by it. Then $\{X(\mathcal{N}(t),t)=\exp\{u\mathcal{N}(t)-(e^u-1)\Lambda(t)\}\}_{t\ge 0}$, where $\Lambda(t)=\int_{o}^{t}\lambda(u)du$ is a $\mathcal{F}_t$- martingale if $\{\mathcal{N}(t)-\Lambda(t)\}_{t\ge 0}$ is a $\mathcal{F}_t$-martingale.  
\end{lemma}
\begin{proof}
    Let  $F(N(t),t)$ be a function of a jump process $\{N(t)\}_{t\ge 0}$ with jump size $+1$. Then
    \begin{align*}
        dF(N(t),t)&=\frac{\partial}{\partial N(t)}F(N(t),t)dN(t)+\frac{\partial}{\partial t}F(N(t),t)dt\\
        &=\left[F(N(t-)+1)-F(N(t-))\right]dN(t)+\frac{\partial}{\partial t}F(N(t),t)dt.
    \end{align*}
    It follows that
    \begin{align}
        dX(\mathcal{N}(t),t)=\left[X(\mathcal{N}(t-)+1,t)-X(\mathcal{N}(t-),t)\right]d\mathcal{N}(t)+\frac{\partial}{\partial t}X(\mathcal{N}(t),t)dt.\label{npp_watanabe1}
    \end{align}
    Now 
    \begin{align}
        X(\mathcal{N}(t-)+1,t)-X(\mathcal{N}(t-),t)&=e^{u\mathcal{N}(t-)+u-(e^u-1)\Lambda(t)}-e^{u\mathcal{N}(t-)-(e^u-1)\Lambda(t)}\nonumber\\
        &=(e^u-1)X(\mathcal{N}(t-),t)\label{npp_watanabe2}
    \end{align}
    and
    \begin{align}
        \frac{\partial}{\partial t}X(\mathcal{N}(t),t)&=X(\mathcal{N}(t),t)\frac{\partial}{\partial t}\left(-(e^u-1)\int_{0}^{\infty}\lambda(u)du\right)\nonumber\\
        &=-(e^u-1)\lambda(t)X(\mathcal{N}(t),t).\label{npp_watanabe3}
    \end{align}
    From \eqref{npp_watanabe1}, \eqref{npp_watanabe2} and \eqref{npp_watanabe3}, we get
    \begin{align}
         dX(\mathcal{N}(t),t)&=(e^u-1)X(\mathcal{N}(t-),t)d\mathcal{N}(t)-(e^u-1)\lambda(t)X(\mathcal{N}(t),t)dt\nonumber\\
         &=(e^u-1)X(\mathcal{N}(t-),t)d\mathcal{N}(t)-(e^u-1)\lambda(t)X(\mathcal{N}(t-),t)dt\nonumber\\
         &+(e^u-1)\lambda(t)X(\mathcal{N}(t-),t)dt-(e^u-1)\lambda(t)X(\mathcal{N}(t),t)dt\nonumber\\
         &=(e^u-1)X(\mathcal{N}(t-),t)\left(d\mathcal{N}(t)-\lambda(t)dt\right)+(e^u-1)\lambda(t)\left[X(\mathcal{N}(t-),t)-X(\mathcal{N}(t),t)\right]dt.
    \end{align}
    On integrating the above equation from $0$ to $t$, we get
    \begin{align}
       X(\mathcal{N}(t),t)-X(\mathcal{N}(0),0)&=(e^u-1)\int_{0}^{t}X(\mathcal{N}(s-),s)\left(d\mathcal{N}(s)-\lambda(s)ds\right)\nonumber\\
       &+(e^u-1)\int_{0}^{t}\lambda(s)\left[X(\mathcal{N}(s-),s)-X(\mathcal{N}(s),s)\right]ds.\nonumber
    \end{align}
    \begin{equation}
        \implies X(\mathcal{N}(t),t)=1+(e^u-1)\int_{0}^{t}X(\mathcal{N}(s-),s)\left(d\mathcal{N}(s)-\lambda(s)ds\right),\nonumber
    \end{equation}
    where the last step follows since $\{X(\mathcal{N}(s),s)\}_{t\ge 0}$ has finitely many jumps which implies $\int_{0}^{t}X(\mathcal{N}(s-),s)ds\thickapprox \int_{0}^{t}X(\mathcal{N}(s),s)ds$. Since $X(\mathcal{N}(t-),t)$ is predictable and $\{\mathcal{N}(t)-\Lambda(t)\}_{t\ge 0}$ is a $\mathcal{F}_t$-martingale, Theorem 11.4.5 of \citet{Shreve2004} guarantees that $\{X(\mathcal{N}(t),t)\}_{t\ge 0}$ is a $\mathcal{F}_t$-martingale (see example 11.5.2 of \citet{Shreve2004}).
\end{proof}
\citet{Watanabe1964} provided the following martingale characterization of a classical homogeneous Poisson process.

\begin{theorem}
Let $\{N(t)\}_{t \ge 0}$ be a $\{\mathcal{F}_t\}_{t \ge 0}$-adapted simple locally finite point process. Then 
$\{N(t) - \lambda t\}_{t \ge 0}$ is a $\{\mathcal{F}_t\}_{t \ge 0}$-martingale for some $\lambda > 0$ if and only if 
$\{N(t)\}_{t \ge 0}$ is a homogeneous Poisson process with intensity $\lambda > 0$.
\end{theorem}

We extend Watanabe's martingale characterization of a Poisson process to a NPP as shown below. 
\begin{theorem}\label{watanabe_npp}
Let $\{\mathcal{N}(t)\}_{t\ge 0}$ be a simple locally finite point process adapted to $\mathcal{F}_t$ where $\{\mathcal{F}_t=\sigma(\mathcal{N}(t))\}_{t\ge 0}$ is the natural filtration of $\{\mathcal{N}(t)\}_{t\ge 0}$. Then $\{\mathcal{N}(t)-\Lambda(t)\}_{t\ge 0}$ is a $\mathcal{F}_t$-martingale for some $\Lambda(t)>0$ if and only if $\{\mathcal{N}(t)\}_{t\ge 0}$ is a NPP with cumulative intensity function $\Lambda(t)$. 
\end{theorem}

\begin{proof}
Let $\{\mathcal{N}(t)\}_{t\ge 0}$ be a NPP with cumulative intensity function $\Lambda(t)>0$. Then for $0\le s<t$, we have
\begin{equation*}
\mathbb{E}(\mathcal{N}(t)-\Lambda(t)\big|\mathcal{F}_s)=\mathbb{E}(\mathcal{N}(t)-\mathcal{N}(s)\big|\mathcal{F}_s)+\mathcal{N}(s)-\Lambda(t)=\mathbb{E}(\mathcal{N}(t)-\mathcal{N}(s))+\mathcal{N}(s)-\Lambda(t)=\mathcal{N}(s)-\Lambda(s),
\end{equation*}
which implies $\{\mathcal{N}(t)-\Lambda(t)\}_{t\ge 0}$ is a $\mathcal{F}_t$-martingale. Next suppose $\{\mathcal{N}(t)-\Lambda(t)\}_{t\ge 0}$ is a $\mathcal{F}_t$-martingale. 
By Lemma \ref{martingale lemma}, $\{X(\mathcal{N}(t),t)=\exp(u\mathcal{N}(t)-(e^u-1)\Lambda(t))\}_{t\ge 0}$ is a $\mathcal{F}_t$-martingale.
Now using the fact that $X(\mathcal{N}(t),t)=X(\mathcal{N}(s),s)\exp\{u[\mathcal{N}(t)-\mathcal{N}(s)]-(e^u-1)\left(\Lambda(t)-\Lambda(s)\right)\}$ and by the definition of a martingale, we have
\begin{align}
    &\mathbb{E}\left(X(\mathcal{N}(t),t)\big|\mathcal{F}_s\right)=X(\mathcal{N}(s),s)\nonumber\\
    \implies&\mathbb{E}\left(X(\mathcal{N}(s),s)e^{u[\mathcal{N}(t)-\mathcal{N}(s)]-(e^u-1)\left(\Lambda(t)-\Lambda(s)\right)}\big|\mathcal{F}_s\right)=X(\mathcal{N}(s),s)\nonumber\\
    \implies &X(\mathcal{N}(s),s)\mathbb{E}\left(e^{u[\mathcal{N}(t)-\mathcal{N}(s)]-(e^u-1)\left(\Lambda(t)-\Lambda(s)\right)}\big|\mathcal{F}_s\right)=X(\mathcal{N}(s),s)\nonumber\\
    \implies &\mathbb{E}\left(e^{u[\mathcal{N}(t)-\mathcal{N}(s)]-(e^u-1)\left(\Lambda(t)-\Lambda(s)\right)}\big|\mathcal{F}_s\right)=1\nonumber\\
    \implies &\mathbb{E}\left(e^{u[\mathcal{N}(t)-\mathcal{N}(s)]}\big|\mathcal{F}_s\right)=e^{(e^u-1)\left(\Lambda(t)-\Lambda(s)\right)}.\label{increment_distribution_npp}
\end{align}
Thus $\mathcal{N}(t)-\mathcal{N}(s)\big|\mathcal{F}_s\sim$ Poi$\left(\Lambda(t)-\Lambda(s)\right)$.
Now taking expectation on both sides of \eqref{increment_distribution_npp}, we get
\begin{align}
    &\mathbb{E}\left[\mathbb{E}\left(e^{u[\mathcal{N}(t)-\mathcal{N}(s)]}\big|\mathcal{F}_s\right)\right]=\mathbb{E}\left(e^{(e^u-1)\left(\Lambda(t)-\Lambda(s)\right)}\right)\nonumber\\
    \implies&\mathbb{E}\left(e^{u[\mathcal{N}(t)-\mathcal{N}(s)]}\right)=e^{(e^u-1)\left(\Lambda(t)-\Lambda(s)\right)},\nonumber
\end{align}
that is, $\mathcal{N}(t)-\mathcal{N}(s)\sim$ Poi$\left(\Lambda(t)-\Lambda(s)\right)$. Taking $s=0$, we have $\mathcal{N}(t)\sim$ Poi$\left(\Lambda(t)\right)$, that is, $\{\mathcal{N}(t)\}_{t\ge 0}$ is a NPP.
\end{proof}

The following result is a generalization of Lemma \ref{martingale lemma} for a weighted sum of $k$ non-homogeneous counting processes.
\begin{proposition}\label{martingale_proposition}
    Let $\{\mathcal{N}_j(t)\}_{t\ge 0}$ be $k$ independent counting processes with jump size $+1$ and rates $\lambda_j(t)$ for $j=1,2,\dots,k$ such that $\mathcal{N}_j(0)=0$ and $X(t)=\sum_{j=1}^{k}uj\mathcal{N}_j(t)-\sum_{j=1}^{k}(e^{uj}-1)\Lambda_j(t)$, where $\Lambda_j(t)=\int_{o}^{t}\lambda_j(u)du$. Then $\{\exp\left(X(t)\right)\}_{t\ge 0}$ is a martingale with respect to $\mathcal{F}_t=\sigma(\{\mathcal{N}_1(s),\mathcal{N}_2(s),\dots,\mathcal{N}_k(s)\},\,s\le t)$, $t\ge 0$ if for each $j$, the process $\{\mathcal{N}_j(t)-\Lambda_j(t)\}_{t\ge 0}$ is a martingale with respect to $\mathcal{F}_t^j=\sigma(\{\mathcal{N}_j(s)\},\,s\le t)$.  
\end{proposition}
\begin{proof}
    Note that $\sum_{j=1}^{k}uj\mathcal{N}_j(t)$ is the jump part and $\sum_{j=1}^{k}(e^{uj}-1)\Lambda_j(t)$ is the continuous part of the process $X(t)$. Let $f(x)=e^x$. Using It\'o-Doeblin formula for a jump process (see page 483, \citet{Shreve2004}), we obtain
    \begin{align}\label{watanabe_proposition1}
        f(X(t))&=f(X(0))-\int_{0}^{t}\sum_{j=1}^{k}\lambda_j(s)(e^{uj}-1)f'(X(s))ds+\sum_{0<s\le t}\left(f(X(s))-f(X(s-))\right).
    \end{align}
    At time $s$, only one process $\mathcal{N}_j(t)$ can jump. If a jump occurs for the $j^{th}$ process, then $f(X(s))-f(X(s-))=(e^{uj}-1)f(X(s-))$, otherwise $f(X(s))-f(X(s-)=0$. So 
    \begin{align*}
        f(X(s))-f(X(s-))=\sum_{j=1}^{k}(e^{uj}-1)f(X(s-))d\mathcal{N}_j(s).
    \end{align*}
    Thus the jump contributions become
    \begin{equation}\label{watanabe_proposition2}
        \sum_{0<s\le t}f(X(s))-f(X(s-))=\int_{0}^{t}\sum_{j=1}^{k}(e^{uj}-1)f(X(s-))d\mathcal{N}_j(s).
    \end{equation}
    From \eqref{watanabe_proposition1} and \eqref{watanabe_proposition2}, we get
    \begin{align}
        f(X(t))&=1-\int_{0}^{t}\sum_{j=1}^{k}\lambda_j(s)(e^{uj}-1)f(X(s))ds+\int_{0}^{t}\sum_{j=1}^{k}(e^{uj}-1)f(X(s-))d\mathcal{N}_j(s)\nonumber\\
        &=1-\int_{0}^{t}\sum_{j=1}^{k}\lambda_j(s)(e^{uj}-1)f(X(s))ds+\int_{0}^{t}\sum_{j=1}^{k}(e^{uj}-1)f(X(s-))d\mathcal{N}_j(s)\nonumber\\
        &-\int_{0}^{t}\sum_{j=1}^{k}(e^{uj}-1)f(X(s-))\lambda_j(s)ds+\int_{0}^{t}\sum_{j=1}^{k}(e^{uj}-1)f(X(s-))\lambda_j(s)ds\nonumber\\
        &=1+\int_{0}^{t}\sum_{j=1}^{k}(e^{uj}-1)f(X(s-))\left[d\mathcal{N}_j(s)-\lambda_j(s)ds\right]-\int_{0}^{t}\sum_{j=1}^{k}\lambda_j(s)(e^{uj}-1)\left(f(X(s))-f(X(s-))\right)ds\nonumber\\
        &=1+\sum_{j=1}^{k}(e^{uj}-1)\int_{0}^{t}f(X(s-))\left[d\mathcal{N}_j(s)-\lambda_j(s)ds\right],
    \end{align}
    where the last step follows because $\{f(X(s))\}_{t\ge 0}$ has finitely many jumps. Now, as $\{f(X(t-))\}_{t\ge 0}$ is predictable and $\{\mathcal{N}_j(t)-\Lambda_j(t)\}_{t\ge 0}$ is a $\mathcal{F}_t^j$-martingale, by Theorem 11.4.5 of \citet{Shreve2004}, $\{\exp\left(X(t)\right)\}_{t\ge 0}$ is a $\mathcal{F}_t$-martingale.
\end{proof}
\subsection{Equivalence of martingales}
Let $\{\mathcal{N}(t)\}_{t \ge 0}$ be a counting process with intensity function $\lambda(t)$ and compensator $\Lambda(t) = \int_0^t \lambda(s)\,ds$. Then $\{\mathcal{N}(t) - \Lambda(t)\}_{t\ge 0}$ and $\left\{\exp\!\left[\theta\,\mathcal{N}(t) - \Lambda(t)(e^{\theta} - 1)\right]\right\}_{t\ge 0}$ for $\theta\in\mathbb{R}$ are called the compensated and the stochastic exponential processes associated with $\mathcal{N}(t)$ respectively. The next result shows the equivalence between the compensated and exponential martingale forms for the NPP.
\begin{theorem}\label{martingale_equivalence}
    For a NPP $\{\mathcal{N}(t)\}_{t\ge 0}$, the process $\{\mathcal{N}(t)-\Lambda(t)\}_{t\ge 0}$ is a $\mathcal{F}_t$-martingale if and only if $\{X(t)=\exp[u\mathcal{N}(t)-(e^u-1)\Lambda(t)]\}_{t\ge 0}$ is a $\mathcal{F}_t$-martingale, where $\mathcal{F}_t=\sigma(\{\mathcal{N}(s)\},\,s\le t)$, $t\ge 0$.
\end{theorem}
\begin{proof}
    Suppose that the process $\{\mathcal{N}(t)-\Lambda(t)\}_{t\ge 0}$ is a martingale. By Lemma \ref{martingale lemma}, it is clear that $\{X(t)=\exp[u\mathcal{N}(t)-(e^u-1)\Lambda(t)]\}_{t\ge 0}$ is a martingale. 
    For the converse part, suppose that $\{X(t)=\exp[u\mathcal{N}(t)-(e^u-1)\Lambda(t)]\}_{t\ge 0}$ is a martingale. Therefore by definition, we have
    \begin{align}
        \mathbb{E}\left(X(t)\big|\mathcal{F}_s\right)=X(s).\label{martingale_equivalence_eq1}
    \end{align}
    For fixed $s$ and $t$ with $0<s\le t$, $X(t)$ is a function of $u$. Since $X(t)$ is smooth in $u$, differentiating \eqref{martingale_equivalence_eq1} with respect to $u$ at $u=0$, we get
    \begin{align}
        &\frac{\partial}{\partial u}\mathbb{E}\left(X(t)\big|\mathcal{F}_s\right)\bigg|_{u=0}=\frac{\partial}{\partial u}X(s)\bigg|_{u=0}\nonumber\\
        \implies&\mathbb{E}\left(\frac{\partial}{\partial u}\exp\left\{u\mathcal{N}(t)-(e^u-1)\Lambda(t)\right\}\big|_{u=0}\bigg|\mathcal{F}_s\right)=\left[\left(\mathcal{N}(s)-e^u\Lambda(s)\right)\exp\left\{u\mathcal{N}(s)-(e^u-1)\Lambda(s)\right\}\right]\bigg|_{u=0}\nonumber\\
        \implies&\mathbb{E}\left(\left(\mathcal{N}(t)-e^u\Lambda(t)\right)\exp\left\{u\mathcal{N}(t)-(e^u-1)\Lambda(t)\right\}\big|_{u=0}\bigg|\mathcal{F}_s\right)=\mathcal{N}(s)-\Lambda(s)\nonumber\\
        \implies&\mathbb{E}\left(\mathcal{N}(t)-\Lambda(t)\big|\mathcal{F}_s\right)=\mathcal{N}(s)-\Lambda(s).\nonumber
    \end{align}
    Since $s$ and $t$ are arbitrary, it follows that $\{\mathcal{N}(t)-\Lambda(t)\}_{t\ge 0}$ is a martingale.
\end{proof}

The following result gives an equivalent martingale characterization of a NPP.
\begin{theorem}\label{exponential_npp}
Let $\{\mathcal{N}(t)\}_{t\ge 0}$ be a simple locally finite point process adapted to $\mathcal{F}_t$ where $\{\mathcal{F}_t=\sigma(\mathcal{N}(t))\}_{t\ge 0}$ is the natural filtration of $\{\mathcal{N}(t)\}_{t\ge 0}$. Then $\{\exp\{u\mathcal{N}(t)-(e^u-1)\Lambda(t)\}\}_{t\ge 0}$ is a $\mathcal{F}_t$-martingale for some $\Lambda(t)>0$ if and only if $\{\mathcal{N}(t)\}_{t\ge 0}$ is a NPP with cumulative intensity function $\Lambda(t)$.
\end{theorem}
\begin{proof}
The result follows using Theorems \ref{watanabe_npp} and \ref{martingale_equivalence}.
\end{proof}
Next, we give a generalization of Theorem \ref{martingale_equivalence} for a weighted sum of $k$ NPPs.
\begin{theorem}\label{martingale_equivalence_ngcp}
    Let $\{\mathcal{N}_j(t)\}_{t\ge 0}$ be $k$ independent NPPs with jump size $+1$ and rates $\lambda_j(t)$ for $j=1,2,\dots,k$ such that $\mathcal{N}_j(0)=0$ and $X(t)=\sum_{j=1}^{k}uj\mathcal{N}_j(t)-\sum_{j=1}^{k}(e^{uj}-1)\Lambda_j(t)$, where $\Lambda_j(t)=\int_{o}^{t}\lambda_j(u)du$. Then $\{\sum_{j=1}^{k}j\mathcal{N}_j(t)-\sum_{j=1}^{k}j\Lambda_j(t)\}_{t\ge 0}$ is a $\mathcal{F}_t$-martingale if and only if $\{\exp\left(X(t)\right)\}_{t\ge 0}$ is a $\mathcal{F}_t$-martingale, where $\mathcal{F}_t=\sigma(\{\mathcal{N}_1(s),\mathcal{N}_2(s),\dots,\mathcal{N}_k(s)\},\,s\le t)$, $t\ge 0$.   
\end{theorem}
\begin{proof}
    Suppose that the process $\{\sum_{j=1}^{k}j\mathcal{N}_j(t)-\sum_{j=1}^{k}j\Lambda_j(t)\}_{t\ge 0}$ is a $\mathcal{F}_t$-martingale. Therefore, by definition, we have
    \begin{equation}
        \mathbb{E}\left(\sum_{j=1}^{k}j\mathcal{N}_j(t)-\sum_{j=1}^{k}j\Lambda_j(t)\bigg|\mathcal{F}_s\right)=\sum_{j=1}^{k}j\mathcal{N}_j(s)-\sum_{j=1}^{k}j\Lambda_j(s)\label{ngcp_martingale_equivalence_eq1}
    \end{equation}
    We will first prove that for each $j$, the process $\{\mathcal{N}_j(t)-\Lambda_j(t)\}_{t\ge 0}$ is a martingale with respect to $\mathcal{F}_t^j=\sigma(\{\mathcal{N}_j(s)\},\,s\le t)$. Since $\mathcal{F}_s^j\subseteq\mathcal{F}_s$, using tower property, we have
    \begin{align}
        &\mathbb{E}\left(\sum_{i=1}^{k}i\mathcal{N}_i(t)-\sum_{i=1}^{k}i\Lambda_i(t)\bigg|\mathcal{F}_s^j\right)=\mathbb{E}\left(\mathbb{E}\left(\sum_{i=1}^{k}i\mathcal{N}_i(t)-\sum_{i=1}^{k}i\Lambda_i(t)\bigg|\mathcal{F}_s\right)\bigg|\mathcal{F}_s^j\right)\nonumber\\
        \implies&\mathbb{E}\left(j\mathcal{N}_j(t)-j\Lambda_j(t)+\sum_{i\neq j}^k\left(i\mathcal{N}_i(t)-i\Lambda_i(t)\right)\bigg|\mathcal{F}_s^j\right)=\mathbb{E}\left(\sum_{i=1}^{k}i\mathcal{N}_i(s)-\sum_{i=1}^{k}i\Lambda_i(s)\bigg|\mathcal{F}_s^j\right)\quad(\text{using }\eqref{ngcp_martingale_equivalence_eq1})\nonumber\\
        \implies&\mathbb{E}\left(j\mathcal{N}_j(t)-j\Lambda_j(t)\bigg|\mathcal{F}_s^j\right)=\mathbb{E}\left(j\mathcal{N}_j(s)-j\Lambda_j(s)\bigg|\mathcal{F}_s^j\right)+\sum_{i\neq j}^k\mathbb{E}\left(\left(i\mathcal{N}_i(s)-i\Lambda_i(s)\right)\bigg|\mathcal{F}_s^j\right)\nonumber\\
        \implies&\mathbb{E}\left(\mathcal{N}_j(t)-\Lambda_j(t)\bigg|\mathcal{F}_s^j\right)=\mathcal{N}_j(s)-\Lambda_j(s)\label{ngcp_martingale_equivalence_eq2}
    \end{align}
    where in the last two steps, we have used the independence of $\{\mathcal{N}_i(t)\}_{t\ge 0}$ for $i\neq j$ with respect to $\mathcal{F}_s^j$. 
    Thus \eqref{ngcp_martingale_equivalence_eq2} proves that for each $j$, $\{\mathcal{N}_j(t)-\Lambda_j(t)\}_{t\ge 0}$ is a $\mathcal{F}_t^j$-martingale.
    Therefore, by Proposition \ref{martingale_proposition}, $\{\exp\left(X(t)\right)\}_{t\ge 0}$ is a $\mathcal{F}_t$-martingale. For the converse part, suppose that $\{\exp\left(X(t)\right)\}_{t\ge 0}$ is a $\mathcal{F}_t$-martingale. By definition, we have
    \begin{equation}
        \mathbb{E}\left(\exp\{X(t)\}\big|\mathcal{F}_s\right)=\exp\{X(s)\}.\label{ngcp_martingale_equivalence_eq4}
    \end{equation}
    For fixed $s$ and $t$ with $0<s\le t$, $\exp\{X(t)\}$ is a function of $u$. Since $\exp\{X(t)\}$ is smooth in $u$, differentiating \eqref{ngcp_martingale_equivalence_eq4} with respect to $u$ at $u=0$, we get
    \begin{align}
       &\frac{\partial}{\partial u}\mathbb{E}\left(\exp\{X(t)\}\big|\mathcal{F}_s\right)\bigg|_{u=0}=\frac{\partial}{\partial u}\exp\{X(s)\}\bigg|_{u=0}\nonumber\\ 
       \implies&\mathbb{E}\left(\frac{\partial}{\partial u}\exp\{\sum_{j=1}^{k}uj\mathcal{N}_j(t)-\sum_{j=1}^{k}(e^{uj}-1)\Lambda_j(t)\}\big|_{u=0}\bigg|\mathcal{F}_s\right)\nonumber\\
       &=\left[\left(\sum_{j=1}^{k}j\mathcal{N}_j(s)-\sum_{j=1}^{k}je^{uj}\Lambda_j(s)\right)\exp\{\sum_{j=1}^{k}uj\mathcal{N}_j(s)-\sum_{j=1}^{k}(e^{uj}-1)\Lambda_j(s)\}\right]\bigg|_{u=0}\nonumber\\ 
       \implies&\mathbb{E}\left(\left(\sum_{j=1}^{k}j\mathcal{N}_j(t)-\sum_{j=1}^{k}je^{uj}\Lambda_j(t)\right)\exp\{\sum_{j=1}^{k}uj\mathcal{N}_j(t)-\sum_{j=1}^{k}(e^{uj}-1)\Lambda_j(t)\}\big|_{u=0}\bigg|\mathcal{F}_s\right)\nonumber\\
       &=\left(\sum_{j=1}^{k}j\mathcal{N}_j(s)-\sum_{j=1}^{k}j\Lambda_j(s)\right)\nonumber\\
       \implies&\mathbb{E}\left(\sum_{j=1}^{k}j\mathcal{N}_j(t)-\sum_{j=1}^{k}j\Lambda_j(t)\bigg|\mathcal{F}_s\right)=\sum_{j=1}^{k}j\mathcal{N}_j(s)-\sum_{j=1}^{k}j\Lambda_j(s).\nonumber
    \end{align}
    Since $s$ and $t$ are arbitrary, it follows that $\{\sum_{j=1}^{k}j\mathcal{N}_j(t)-\sum_{j=1}^{k}j\Lambda_j(t)\}_{t\ge 0}$ is a $\mathcal{F}_t$-martingale.
\end{proof}

\subsection{Martingale characterization of NGCP}
We begin with the following result that shows the unique representation of a NGCP as a weighted sum of NPPs where the weights are the jump sizes of the NGCP. 
\begin{theorem}\label{ngcp_npp_equality}
    Let $\{\mathcal{N}_1(t)\}_{t\ge 0}$, $\{\mathcal{N}_2(t)\}_{t\ge 0}$, $\dots\{\mathcal{N}_k(t)\}_{t\ge 0}$ be $k$ independent counting processes such that $\mathcal{N}_j(0)=0$ for $j=1,2,\dots,k$. Then the weighted process $\{\sum_{j=1}^{k}j\mathcal{N}_j(t)\}_{t\ge 0}$ is a NGCP if and only if each $\{\mathcal{N}_j(t)\}_{t\ge 0}$ is a non-homogeneous Poisson process (NPP) with rate $\lambda_j(t)$.
\end{theorem}
\begin{proof}
    Suppose that $\{\mathcal{N}_1(t)\}_{t\ge 0}$, $\{\mathcal{N}_2(t)\}_{t\ge 0}$, $\dots\{\mathcal{N}_k(t)\}_{t\ge 0}$ are $k$ independent NPPs. By Proposition 3.2 of \citet{Tathe2024}, the process $\{\sum_{j=1}^{k}j\mathcal{N}_j(t)\}_{t\ge 0}$ is equal in distribution to a NGCP. For $i\in \mathbb{N}$, consider
    \begin{align}
        \Delta_i\left(\sum_{j=1}^{k}j\mathcal{N}_j(t)\right)=\sum_{j=1}^{k}j\mathcal{N}_j(t_i)-\sum_{j=1}^{k}j\mathcal{N}_j(t_{i-1})=\sum_{j=1}^{k}j\left(\mathcal{N}_j(t_i)-\mathcal{N}_j(t_{i-1})\right).\nonumber
    \end{align}
    Since $\{\mathcal{N}_j(t)\}_{t\ge 0}$ are independent of each other and for fixed $j$, the random variables $\{\Delta_i(\mathcal{N}_j(t))=\mathcal{N}_j(t_i)-\mathcal{N}_j(t_{i-1})\}$ are independent, the collection $\{\Delta_i(\mathcal{N}_j(t)), 1\le j\le k\}$ is independent. As $\Delta_i\left(\sum_{j=1}^{k}j\mathcal{N}_j(t)\right)$ is a linear combination of independent random variables, the collection $\{\Delta_i\left(\sum_{j=1}^{k}j\mathcal{N}_j(t)\right)$, $i\in\mathbb{N}\}$ is independent, i.e., the process $\{\sum_{j=1}^{k}j\mathcal{N}_j(t)\}_{t\ge 0}$ has independent increments. This proves that $\{\sum_{j=1}^{k}j\mathcal{N}_j(t)\}_{t\ge 0}$ is a NGCP. 
    
    We will prove the converse part using the method of mathematical induction. For $k=1$, the result holds true as $\{\mathcal{N}_1(t)\}_{t\ge o}$ is a NPP with rate $\lambda_1(t)$. 
    
    Assume that the result holds true for $k=m$. That is, if the weighted process $\{\sum_{j=1}^{m}j\mathcal{N}_j(t)\}_{t\ge 0}$ is a NGCP, then each process $\{\mathcal{N}_j(t)\}_{t\ge 0}$ is a non-homogeneous Poisson process with rate $\lambda_j(t)$ for $j=1,2,\dots,m$. 
    
    We will now prove the result for $k=m+1$. Consider $m+1$ independent counting processes $\{\mathcal{N}_1(t)\}_{t\ge 0}$, $\{\mathcal{N}_2(t)\}_{t\ge 0}$, $\dots\{\mathcal{N}_{m+1}(t)\}_{t\ge 0}$ with rates $\lambda_1(t),\lambda_2(t),\dots,\lambda_{m+1}(t)$ respectively such that the weighted process $\{\sum_{j=1}^{m+1}j\mathcal{N}_j(t)\}_{t\ge 0}$ is a NGCP. We have
    \begin{equation}
        \mathbb{E}\left(e^{u\sum_{j=1}^{m+1}j\mathcal{N}_{j}(t)}\right)=\mathbb{E}\left(e^{u(m+1)\mathcal{N}_{m+1}(t)}\right)\prod_{j=1}^{m}e^{\Lambda_j(t)(e^{uj}-1)}\label{weighted_sum_equality_eq1}
    \end{equation}
    since $\{\mathcal{N}_1(t)\}_{t\ge 0}$, $\{\mathcal{N}_2(t)\}_{t\ge 0}$, $\dots\{\mathcal{N}_{m}(t)\}_{t\ge 0}$ are independent NPPs by the induction hypothesis. The moment generating function of the NGCP $\{\sum_{j=1}^{m+1}j\mathcal{N}_j(t)\}_{t\ge 0}$ is given by
    \begin{equation}
        \mathbb{E}\left(e^{u\sum_{j=1}^{m+1}j\mathcal{N}_{j}(t)}\right)=e^{\sum_{j=1}^{m+1}\Lambda_j(t)(e^{uj}-1)}.\label{weighted_sum_equality_eq2}
    \end{equation}
    Deviding \eqref{weighted_sum_equality_eq2} by \eqref{weighted_sum_equality_eq1}, we get
    \begin{equation*}
        \mathbb{E}\left(e^{u(m+1)\mathcal{N}_{m+1}(t)}\right)=e^{\Lambda_{m+1}(t)(e^{uj}-1)}.
    \end{equation*}
    Thus $\{\mathcal{N}_{m+1}(t)\}_{t\ge 0}$ is equal in distribution to a NPP. We will now show that $\{\mathcal{N}_{m+1}(t)\}_{t\ge 0}$ has independent increments.
    
    Let $0\le t_0<\ldots<t_r<\infty$. Due to the independence of $m+1$ counting processes, we have
    \begin{align}
        &\mathbb{E}\left(\exp\left(\sum_{i=1}^{r}u_i\sum_{j=1}^{m+1}j\left(\mathcal{N}_{j}(t_i)-\mathcal{N}_{j}(t_{i-1})\right)\right)\right)\nonumber\\
        &=\prod_{i=1}^{r}\mathbb{E}\left(\exp\left(u_i\sum_{j=1}^{m}j\left(\mathcal{N}_{j}(t_i)-\mathcal{N}_{j}(t_{i-1})\right)\right)\right)\mathbb{E}\left(e^{u_i(m+1)\left(\mathcal{N}_{m+1}(t_i)-\mathcal{N}_{m+1}(t_{i-1})\right)}\right)\nonumber\\
        &=\prod_{i=1}^{r}\prod_{j=1}^{m}e^{u_i\left[\Lambda_j(t_i)-\Lambda_j(t_{i-1})\right](e^{u_ij}-1)} \mathbb{E}\left(e^{u_i(m+1)\left(\mathcal{N}_{m+1}(t_i)-\mathcal{N}_{m+1}(t_{i-1})\right)}\right)\label{weighted_sum_equality_eq3}
    \end{align}
    where the last step is obtained using the distribution of increments of $\{\mathcal{N}_{j}(t)\}_{t\ge 0}$ for $1\le j\le m$ which are independent NPPs. Now, using the independent increments of the NGCP $\{\sum_{j=1}^{m+1}j\mathcal{N}_j(t)\}_{t\ge 0}$, we have
    \begin{align}
        \mathbb{E}\left(\exp\left(\sum_{i=1}^{r}u_i\sum_{j=1}^{m+1}j\left(\mathcal{N}_{j}(t_i)-\mathcal{N}_{j}(t_{i-1})\right)\right)\right)&=\prod_{j=1}^{r}\mathbb{E}\left(\exp\left(u_i\sum_{j=1}^{m+1}j\left(\mathcal{N}_{j}(t_i)-\mathcal{N}_{j}(t_{i-1})\right)\right)\right)\nonumber\\
        &=\prod_{i=1}^{r}e^{u_i\sum_{j=1}^{m+1}\left[\Lambda_j(t_i)-\Lambda_j(t_{i-1})\right](e^{u_ij}-1)}\label{weighted_sum_equality_eq4}
    \end{align}
    Dividing \eqref{weighted_sum_equality_eq4} by \eqref{weighted_sum_equality_eq3}, we get
    \begin{align}
        &\prod_{i=1}^{r}e^{u_i\left[\Lambda_{m+1}(t_i)-\Lambda_{m+1}(t_{i-1})\right](e^{u_i(m+1)}-1)}=\prod_{i=1}^{r} \mathbb{E}\left(e^{u_i(m+1)\left(\mathcal{N}_{m+1}(t_i)-\mathcal{N}_{m+1}(t_{i-1})\right)}\right)\nonumber\\
        \implies&e^{\sum_{i=1}^{r}u_i\left[\Lambda_{m+1}(t_i)-\Lambda_{m+1}(t_{i-1})\right](e^{u_i(m+1)}-1)}=\prod_{i=1}^{r} \mathbb{E}\left(e^{u_i(m+1)\left(\mathcal{N}_{m+1}(t_i)-\mathcal{N}_{m+1}(t_{i-1})\right)}\right)\nonumber\\
        \implies&\mathbb{E}\left(e^{\sum_{i=1}^{r}u_i(m+1)\left[\mathcal{N}_{m+1}(t_i)-\mathcal{N}_{m+1}(t_{i-1})\right]}\right)=\prod_{i=1}^{r} \mathbb{E}\left(e^{u_i(m+1)\left(\mathcal{N}_{m+1}(t_i)-\mathcal{N}_{m+1}(t_{i-1})\right)}\right),\nonumber
    \end{align}
    which implies the counting process $\{\mathcal{N}_{m+1}(t)\}_{t\ge 0}$ has independent increments. Thus $\{\mathcal{N}_{m+1}(t)\}_{t\ge 0}$ is a NPP with rate $\lambda_{m+1}(t)$. This proves the theorem. 
\end{proof}
The next result provides the martingale characterization of a NGCP.
\begin{theorem}\label{martingale_characterization_exponential_ngcp}
    Let $\{\mathcal{N}_j(t)\}_{t\ge 0}$ be $k$ independent counting processes with jump size $+1$ and rates $\lambda_j(t)$ for $j=1,2,\dots,k$ such that $\mathcal{N}_j(0)=0$ and $X(t)=\sum_{j=1}^{k}uj\mathcal{N}_j(t)-\sum_{j=1}^{k}(e^{uj}-1)\Lambda_j(t)$, where $\Lambda_j(t)=\int_{o}^{t}\lambda_j(u)du$. Then $\{\exp\left(X(t)\right)\}_{t\ge 0}$ is a $\mathcal{F}_t$-martingale for $\mathcal{F}_t=\sigma(\{\mathcal{N}_1(s),\mathcal{N}_2(s),\dots,\mathcal{N}_k(s)\},\,s\le t)$, $t\ge 0$ if and only if the weighted process $\{\sum_{j=1}^{k}j\mathcal{N}_j(t)\}_{t\ge 0}$ is a NGCP.
\end{theorem}
\begin{proof}
    Suppose that $\{\exp\left(X(t)\right)=\exp\left(\sum_{j=1}^{k}uj\mathcal{N}_j(t)-\sum_{j=1}^{k}(e^{uj}-1)\Lambda_j(t)\right)\}_{t\ge 0}$ is a $\mathcal{F}_t$-martingale. Then we have $\mathbb{E}\left(\exp\left(X(t)\right)\right)=\mathbb{E}\left(\exp\left(X(0)\right)\right)=1$. Thus
    \begin{equation}
        \mathbb{E}\left(\exp\left(\sum_{j=1}^{k}uj\mathcal{N}_j(t)\right)\right)=\exp\left(\sum_{j=1}^{k}(e^{uj}-1)\Lambda_j(t)\right),\,u\in\mathbb{R}
    \end{equation}
    which implies that the weighted process $\{\sum_{j=1}^{k}j\mathcal{N}_j(t)\}_{t\ge 0}$ is equal in distribution to a NGCP. Again using the martingale property of $\{\exp\left(X(t)\right)\}_{t\ge 0}$, we have for $0 < s \le t$
    \begin{align}
        \mathbb{E}\left(\exp\left(\sum_{j=1}^{k}uj\mathcal{N}_j(t)-\sum_{j=1}^{k}(e^{uj}-1)\Lambda_j(t)\right)\bigg|\mathcal{F}_s\right)&=\exp\left(\sum_{j=1}^{k}uj\mathcal{N}_j(s)-\sum_{j=1}^{k}(e^{uj}-1)\Lambda_j(s)\right)\nonumber\\
        \implies\mathbb{E}\left(\exp\left(u\sum_{j=1}^{k}j(\mathcal{N}_j(t)-\mathcal{N}_j(s))\right)\bigg|\mathcal{F}_s\right)&=\exp\left(\sum_{j=1}^{k}(\Lambda_j(t)-\Lambda_j(s))(e^{uj}-1)\right).\label{independent_increments_martingale_ngcp1}
    \end{align}
    Taking expectation on both sides of \eqref{independent_increments_martingale_ngcp1}, we get
    \begin{equation}
    \mathbb{E}\left(\exp\left(u\sum_{j=1}^{k}j(\mathcal{N}_j(t)-\mathcal{N}_j(s))\right)\right)=\mathbb{E}\left(\exp\left(\sum_{j=1}^{k}(\Lambda_j(t)-\Lambda_j(s))(e^{uj}-1)\right)\right).\label{independent_increments_martingale_ngcp2}
    \end{equation}
    We will now show that the weighted process $\{\sum_{j=1}^{k}j\mathcal{N}_j(t)\}_{t\ge 0}$ has independent increments.
    
    Let $0\le t_0<\ldots<t_n<\infty$. Consider
    \begin{align}
        &\mathbb{E}\left(\exp\left(\sum_{i=1}^{n}u_i\sum_{j=1}^{k}j(\mathcal{N}_j(t_i)-\mathcal{N}_j(t_{i-1}))\right)\right)\nonumber\\
        &=\mathbb{E}\left(\mathbb{E}\left(\exp\left(\sum_{i=1}^{n}u_i\sum_{j=1}^{k}j(\mathcal{N}_j(t_i)-\mathcal{N}_j(t_{i-1}))\right)\bigg|\mathcal{F}_{t_{n-1}}\right)\right)\nonumber\\
        &=\mathbb{E}\left(\exp\left(\sum_{i=1}^{n-1}u_i\sum_{j=1}^{k}j(\mathcal{N}_j(t_i)-\mathcal{N}_j(t_{i-1}))\right)\times\mathbb{E}\left(\exp\left(u_n\sum_{j=1}^{k}j(\mathcal{N}_j(t_n)-\mathcal{N}_j(t_{n-1}))\right)\bigg|\mathcal{F}_{t_{n-1}}\right)\right)\nonumber\\
        &=\mathbb{E}\left(\exp\left(\sum_{i=1}^{n-1}u_i\sum_{j=1}^{k}j(\mathcal{N}_j(t_i)-\mathcal{N}_j(t_{i-1}))\right)\right)\times\exp\left(\sum_{j=1}^{k}(\Lambda_j(t_n)-\Lambda_j(t_{n-1}))(e^{u_nj}-1)\right)\quad(\text{using}~\eqref{independent_increments_martingale_ngcp1})\nonumber\\
        &=\mathbb{E}\left(\exp\left(\sum_{i=1}^{n-1}u_i\sum_{j=1}^{k}j(\mathcal{N}_j(t_i)-\mathcal{N}_j(t_{i-1}))\right)\right)\times\mathbb{E}\left(\exp\left(u_n\sum_{j=1}^{k}j(\mathcal{N}_j(t_n)-\mathcal{N}_j(t_{n-1}))\right)\right)\quad(\text{using}~\eqref{independent_increments_martingale_ngcp2}).\nonumber
    \end{align}
 By applying the conditioning argument successively with respect to the filtrations $\mathcal{F}_{t_{n-2}}, \mathcal{F}_{t_{n-3},}\dots,\mathcal{F}_{t_{0}}$, the above procedure can be iterated. This yields, inductively, 
    \begin{align}
        \mathbb{E}\left(\exp\left(\sum_{i=1}^{n}u_i\sum_{j=1}^{k}j(\mathcal{N}_j(t_i)-\mathcal{N}_j(t_{i-1}))\right)\right)=\prod_{i=1}^{n}\mathbb{E}\left(\exp\left(u_i\sum_{j=1}^{k}j(\mathcal{N}_j(t_i)-\mathcal{N}_j(t_{i-1}))\right)\right)\nonumber
    \end{align}
    which implies the counting process $\{\sum_{j=1}^{k}j\mathcal{N}_j(t)\}_{t\ge 0}$ has independent increments. Thus $\{\sum_{j=1}^{k}j\mathcal{N}_j(t)\}_{t\ge 0}$ is a NGCP. Alternatively, we can show that $\{\sum_{j=1}^{k}j\mathcal{N}_j(t)\}_{t\ge 0}$ is a NGCP as follows. 
    
    Since $\{\exp({X(t)})\}_{t\ge 0}$ is a martingale, we have
    \begin{align}
        &\mathbb{E}\left(\exp\left(\sum_{j=1}^{k}uj\mathcal{N}_j(t)-\sum_{j=1}^{k}(e^{uj}-1)\Lambda_j(t)\right)\bigg|\mathcal{F}_s\right)=\exp\left(\sum_{j=1}^{k}uj\mathcal{N}_j(s)-\sum_{j=1}^{k}(e^{uj}-1)\Lambda_j(s)\right)\nonumber\\
        \implies&\mathbb{E}\left(\frac{\exp\left(\sum_{j=1}^{k}uj\mathcal{N}_j(t)-\sum_{j=1}^{k}(e^{uj}-1)\Lambda_j(t)\right)}{\exp\left(\sum_{j=1}^{k}uj\mathcal{N}_j(s)-\sum_{j=1}^{k}(e^{uj}-1)\Lambda_j(s)\right)}\bigg|\mathcal{F}_s\right)=1\nonumber\\
        \implies&\mathbb{E}\left(e^{u\sum_{j=1}^{k}j\left(\mathcal{N}_j(t)-\mathcal{N}_j(s)\right)}\big|\mathcal{F}_s\right)=e^{\sum_{j=1}^{k}(e^{uj}-1)(\Lambda_j(t)-\Lambda_j(s))}\nonumber\\
        \implies&\sum_{j=1}^{k}j\left(\mathcal{N}_j(t)-\mathcal{N}_j(s)\right)\bigg|\mathcal{F}_s\sim \text{NGCP}\,\, \text{with cumulative rates }\left(\Lambda_j(t)-\Lambda_j(s)\right) \text{for }1\le j\le k.\label{ngcp_conditional_eq}
    \end{align}
    Putting $s=0$ in \eqref{ngcp_conditional_eq}, we conclude that $\{\sum_{j=1}^{k}j\mathcal{N}_j(t)\}_{t\ge 0}$ is equal in distribution to a NGCP with cumulative rates $\Lambda_j(t)$ for $1\le j\le k$. Also, it is clear from \eqref{ngcp_conditional_eq} that $\{\sum_{j=1}^{k}j\mathcal{N}_j(t)\}_{t\ge 0}$ has independent increments since $s$ and $t$ are arbitrary. Hence $\{\sum_{j=1}^{k}j\mathcal{N}_j(t)\}_{t\ge 0}$ is a NGCP.    
     
    For the converse part, assume that the weighted process $\{\sum_{j=1}^{k}j\mathcal{N}_j(t)\}_{t\ge 0}$ is a NGCP. Using Theorem \ref{ngcp_npp_equality} along with Theorem \ref{watanabe_npp} and Proposition \ref{martingale_proposition}, we see that $\{\exp\left(X(t)\right)\}_{t\ge 0}$ is a $\mathcal{F}_t$-martingale. This completes the proof.
\end{proof}
Now we provide an equivalent martingale characterization of NGCP as follows.
\begin{theorem}\label{martingale_characterization_compensated_ngcp}
    Let $\{\mathcal{N}_j(t)\}_{t\ge 0}$ be $k$ independent counting processes with jump size $+1$ and rates $\lambda_j(t)$ for $j=1,2,\dots,k$ such that $\mathcal{N}_j(0)=0$. Then $\{\sum_{j=1}^{k}j\mathcal{N}_j(t)-\sum_{j=1}^{k}j\Lambda_j(t)\}_{t\ge 0}$, where $\Lambda_j(t)=\int_{o}^{t}\lambda_j(u)du$ is a $\mathcal{F}_t$-martingale for $\mathcal{F}_t=\sigma(\{\mathcal{N}_1(s),\mathcal{N}_2(s),\dots,\mathcal{N}_k(s)\},\,s\le t)$, $t\ge 0$ if and only if the weighted process $\{\sum_{j=1}^{k}j\mathcal{N}_j(t)\}_{t\ge 0}$ is a NGCP.
\end{theorem}

\begin{proof}
The result follows using Theorems \ref{martingale_equivalence_ngcp} and \ref{martingale_characterization_exponential_ngcp}.     
\end{proof}

\section{Martingale characterizations of fractional variants}\label{sec 4}
\citet{Tathe2025} obtained martingale characterizations for some fractional homogeneous processes. In this section, we investigate martingale characterizations for various time-changed variants of the non-homogeneous processes considered in Section \ref{sec 3} and their Skellam versions. First, note that the martingale characterization of a NGCP (see Theorems \ref{martingale_characterization_exponential_ngcp} and \ref{martingale_characterization_compensated_ngcp}) assumes that the underlying NPPs are independent. However, for fractional variants of NGCP, independence among NPPs is not preserved after time changing them by a common stable and/or inverse stable subordinator. So a different approach is required to establish martingale characterizations in such cases. We obtain the compensated martingale characterizations for these fractional variants. The exponential martingale characterizations follow from their equivalence (see Theorems \ref{martingale_equivalence} and \ref{martingale_equivalence_ngcp}), which is independent of any fractional parameter. We begin with the following lemma (see Lemma 1 of \citet{Aletti2018}) which will be used later.
\begin{lemma}\label{martingale_lemma}
    Let $X$ be a right-continuous martingale. If $T$ and $S$ are stopping times such that $P(T<\infty)=1$ and $\{X(t\wedge T),t\ge 0\}$ is uniformly integrable, then $\mathbb{E}(X(T)|\mathcal{F}_{S\wedge T})=X(S\wedge T)$.
\end{lemma}
The following result provides the Watanabe characterization of a non-homogeneous time fractional Poisson process (NTFPP) which is obtained by time changing a NPP with an independent inverse stable subordinator.
\begin{theorem}\label{Watanabe_characterization_NFPP}
    Let $\{K(t)\}_{t\ge 0}$ be a simple locally finite point process. Then $\{K(t)\}_{t\ge 0}$ is a NTFPP if and only if there exist a rate function $\lambda:[0,\infty)\to[0,\infty)$ with $\Lambda(t)=\int_{0}^{t}\lambda(u)du<\infty$ and an inverse stable subordinator $\{Y_{\alpha}(t)\}_{t\ge 0}$ such that the process 
    \begin{equation}\label{martingale_form_nfpp}
    \{X(t)\}_{t\ge 0}=\{K(t)-\Lambda(Y_{\alpha}(t))\}_{t\ge 0}    
    \end{equation} 
    is a right-continuous martingale with respect to the induced filtration $\mathcal{F}_t=\sigma\left(K(s),s\le t\right)\vee\sigma\left(Y_{\alpha}(s), s\ge 0\right)$ and for any $T>0$,
    \begin{equation}\label{stopping_times_martingale_nfpp}
        \{X(\tau),\,\tau\,\, \text{is stopping time such that}\,\,Y_{\alpha}(\tau)\le T\}
    \end{equation}
    is uniformly integrable.
\end{theorem}
\begin{proof}
    Let $\{K(t)\}_{t\ge 0}$ be a NTFPP so that $K(t)=\mathcal{N}\left(Y_{\alpha}(t)\right)$, where $\{Y_{\alpha}(t)\}_{t\ge 0}$ is an inverse stable subordinator independent of the NPP $\{\mathcal{N}(t)\}_{t\ge 0}$ with rate function $\lambda(t)$. Then
\begin{equation*}
    X(t):=K(t)-\Lambda\left(Y_\alpha(t)\right)=\mathcal{N}\left(Y_\alpha(t)\right)-\Lambda\left(Y_\alpha(t)\right).
\end{equation*}
First we show that $\{X(t)\}_{t\ge 0}$ is a martingale with respect to $\mathcal{F}_t=\sigma\left(K(s),s\le t\right)\vee\sigma\left(Y_{\alpha}(s), s\ge 0\right)$. For \(0\le s<t\), using the tower property and conditioning on $Y_\alpha(t)$, we have
\begin{equation*}
    E\left[X(t)\mid\mathcal F_s\right]
=E\left[\,E\left[\mathcal{N}\left(Y_\alpha(t)\right)-\Lambda\left(Y_\alpha(t)\right)\;\big|\;Y_\alpha(t)\vee\mathcal F_s\right]\;\Big|\;\mathcal F_s\right].
\end{equation*}
Note that $\mathcal{N}(Y_\alpha(t))|Y_\alpha(t)\sim \text{Poi}(\Lambda(Y_\alpha(t)))$. 
Using the martingale property of a non-homogeneous Poisson process (see Theorem \ref{watanabe_npp}), we get
\[
E\big[\mathcal{N}\big(Y_\alpha(t)\big)-\Lambda\big(Y_\alpha(t)\big)\;\big|\;Y_\alpha(t)\vee\mathcal F_s\big]
= \mathcal{N}\big(Y_\alpha(s)\big)-\Lambda\big(Y_\alpha(s)\big).
\]
Therefore
\[
E\big[X(t)\mid\mathcal F_s\big]=\mathcal{N}\big(Y_\alpha(s)\big)-\Lambda\big(Y_\alpha(s)\big)=X(s),
\]
which proves $\{X(t)\}_{t\ge 0}$ is a martingale. Right-continuity follows from right-continuity of $\mathcal{N}(t)$ and $Y_\alpha(t)$.
   
 Next we show that the family \(\{X(\tau):Y_\alpha(\tau)\le T\}\) is uniformly integrable. Fix \(T>0\) and let \(\tau\) be any stopping time with \(Y_\alpha(\tau)\le T\). Note that  $\mathbb{E}\left[\mathcal{N}(Y_\alpha(\tau))\right]=\mathbb{E}\left[\Lambda(Y_\alpha(\tau))\right]$, which implies 
 \begin{equation}\label{mean_nfpp_martingale}
     \mathbb{E}\left[X(\tau)\right]=\mathbb{E}\left[\mathcal{N}(Y_\alpha(\tau))-\Lambda(Y_\alpha(\tau))\right]=0.
 \end{equation}
Since $Y_{\alpha}(\tau) \le T$ and $\Lambda(t)$ is non-decreasing in $t$, we have $\Lambda\!\left(Y_{\alpha}(\tau)\right) \le \Lambda(T)$. Given that \(\{\mathcal{N}(t)\}_{t\ge 0}\) is a NPP with intensity function \(\lambda(t)\), and that \(\widetilde{X}(t)=X(\tau\wedge t)\) is a martingale bounded in \(L^2\) with \(\widetilde{X}(0)=0\) a.s., it follows that \(\widetilde{X}(t)\) converges to \(X(\tau)\) in \(L^2\). Using the law of total variance, we have 
\begin{align*}
    \mathbb{V}\left(X(\tau)\right)&=\mathbb{E}\left(\mathbb{V}\left[K(\tau)-\Lambda(Y_{\alpha}(\tau))\mid Y_\alpha(\tau)\right]\right)
+\mathbb{V}\left(\mathbb{E}[K(\tau)-\Lambda(Y_{\alpha}(\tau))\mid Y_\alpha(\tau)]\right)\\
\implies\mathbb{V}\left(X(\tau)\right)&=\mathbb{E}\left(\mathbb{V}\left[\mathcal{N}(Y_\alpha(\tau))\mid Y_\alpha(\tau)\right]\right)+\mathbb{V}\left(\mathbb{E}[\mathcal{N}(Y_\alpha(\tau))\mid Y_\alpha(\tau)]-\Lambda(Y_{\alpha}(\tau))\right)\\
    \implies\mathbb{E}\left[X^2(\tau)\right]
&=\mathbb{E}\left(\mathbb{V}\left[\mathcal{N}(Y_\alpha(\tau))\mid Y_\alpha(\tau)\right]\right)+\left(\mathbb{E}\left[X(\tau)\right]\right)^2\quad(\because\mathcal{N}(Y_\alpha(\tau))|Y_\alpha(\tau)\sim \text{Poi}(\Lambda(Y_\alpha(\tau))))\\
&=\mathbb{E}\left(\mathbb{V}\left[\mathcal{N}(Y_\alpha(\tau))\mid Y_\alpha(\tau)\right]\right)\quad(\text{using}~\eqref{mean_nfpp_martingale})\\
&=\mathbb{E}\left(\Lambda[Y_\alpha(\tau)]\right)\le \Lambda(T).
\end{align*}
Therefore the family \eqref{stopping_times_martingale_nfpp} is uniformly bounded in $L^2$, and hence uniformly integrable.

   Conversely, it is enough to prove that $K(t)=\mathcal{N}\left(Y_{\alpha}(t)\right)$, where $\{\mathcal{N}(t)\}_{t\ge 0}$ is a NPP, independent of the inverse stable subordinator $\{Y_{\alpha}(t)\}_{t\ge 0}$ with rate $\lambda(t)>0$. Let $Z(t)=\inf\{s:Y_{\alpha}(s)\ge t\}$ be the inverse of $\{Y_{\alpha}(t)\}_{t\ge 0}$. It can be observed that $\{Z(t)\}_{t\ge 0}$ forms a family of stopping times. Therefore, by Lemma \ref{martingale_lemma}, 
   \begin{equation*}
       X\left(Z(t)\right)=K(Z(t))-\Lambda\left( Y_{\alpha}(Z(t))\right)
   \end{equation*} is still a martingale. As $Y_{\alpha}(.)$ is continuous, we have $Y_{\alpha}(Z(t))=t$ which implies that $\{K(Z(t))-\Lambda(t)\}_{t\ge 0}$ is a martingale. Furthermore, as $Z(t)$ is an increasing process, $K(Z(t))$ is a simple point process. By Theorem \ref{watanabe_npp}, it follows that $K(Z(t))$ is a classical NPP with rate $\lambda(t)>0$. If we define $\mathcal{N}(t)=K(Z(t))$, then the process $\{K(t)=\mathcal{N}(Y_{\alpha}(t))\}_{t\ge 0}$ represents a NTFPP. This proves the theorem.
\end{proof}
A non-homogeneous space fractional Poisson process (NSFPP) is obtained by time changing a NPP with an independent stable subordinator. Since the moments of the NSFPP do not exist, the Watanabe characterization does not apply here. So we will provide the Watanabe characterization of a non-homogeneous tempered space fractional Poisson process (NTSFPP) $\{\mathcal{N}_{\beta,\theta}(t)\}_{t\ge 0}$ and a non-homogeneous tempered space time fractional Poisson process (NTSTFPP) $\{\mathcal{N}^{\alpha}_{\beta,\theta}(t)\}_{t\ge 0}$, which have finite moments and are defined as
\begin{equation*}
    \mathcal{N}_{\beta,\theta}(t):=\mathcal{N}\left(D_{\beta,\theta}(t)\right),\qquad \mathcal{N}^{\alpha}_{\beta,\theta}(t):=\mathcal{N}\left(D_{\beta,\theta}(Y_{\alpha}(t))\right), 
\end{equation*}
where $\{D_{\beta,\theta}(t)\}_{t\ge 0}$ is the tempered stable subordinator (TSS) (see Section \ref{TSS}) and $\{Y_\alpha(t)\}_{t\ge 0}$ is an inverse stable subordinator which are independent of each other and also independent of the NPP $\{\mathcal{N}(t)\}_{t\ge 0}$.
\begin{theorem}\label{Watanabe_characterization_NTSTFPP}
    Let $\{L(t)\}_{t\ge 0}$ be a simple locally finite point process. Then $\{L(t)\}_{t\ge 0}$ is NTSTFPP if and only if there exist a rate function $\lambda:[0,\infty)\to[0,\infty)$ with $\Lambda(t)=\int_{0}^{t}\lambda(u)du<\infty$, a tempered stable subordinator (TSS) $\{D_{\beta,\theta}(t)\}_{t\ge 0}$ and an inverse stable subordinator $\{Y_\alpha(t)\}_{t\ge 0}$ which are independent of each other such that the process 
    \begin{equation*}\label{martingale_form_NTSTFPP}
    \{R(t)\}_{t\ge 0}=\{L(t)-\Lambda(D_{\beta,\theta}(Y_{\alpha}(t)))\}_{t\ge 0}    
    \end{equation*} 
    is a right-continuous martingale with respect to the induced filtration $\mathcal{F}_t=\sigma\left(L(s),s\le t\right)\vee\sigma\left(D_{\beta,\theta}(Y_{\alpha}(s)), s\ge 0\right)$ and for any $T>0$,
    \begin{equation*}\label{stopping_times_martingale_NTSTFPP}
        \{R(\tau),\,\tau\,\, \text{is stopping time such that}\,\,D_{\beta,\theta}(Y_{\alpha}(\tau))\le T\}
    \end{equation*}
    is uniformly integrable.
\end{theorem}
\begin{proof}   
The proof is similar to that of Theorem \ref{Watanabe_characterization_NFPP} with $Y_{\alpha}(t)$ replaced by $D_{\beta,\theta}(Y_{\alpha}(t))$ and hence omitted.
\end{proof}
The martingale characterization of a NTSFPP follows from Theorem
\ref{Watanabe_characterization_NTSTFPP} by taking $\alpha=1$ and observing that $Y_{1}(t)=t$ a.s.
Next, we investigate martingale characterizations of the generalized versions of NTFPP, NTSFPP and NTSTFPP which allow multiple arrivals with different rate functions. The next result gives the martingale characterization of a non-homogeneous generalized fractional counting process (NGFCP), the generalized version of NTFPP, defined as
\begin{equation}\label{definition_NGFCP}
    \mathcal{M}^{\alpha}(t):=\mathcal{M}(Y_{\alpha}(t)),
\end{equation}
where $\{\mathcal{M}(t)\}_{t\ge 0}$ is a NGCP and $\{Y_\alpha(t)\}_{t\ge 0}$ is an independent inverse stable subordinator.
\begin{theorem}\label{Watanabe_characterization_NGFCP}
    The process $\{\mathcal{M}^\alpha(t)\}_{t\ge 0}$ is a NGFCP if and only if there exist $k$ simple locally finite point processes $\{\mathcal{N}_j^\alpha(t)\}_{t\ge 0}$, $1\le j\le k$ such that $\mathcal{M}^\alpha(t)=\sum_{j=1}^{k}j\mathcal{N}_j^\alpha(t)$ and $X_j(t)=\mathcal{N}_j^\alpha(t)-\Lambda_j(Y_\alpha(t))$, $t\ge 0$ is a right continuous $\{\mathcal{F}_t^j\}_{t\ge 0}$-martingale, where $\Lambda_j(t)=\int_{0}^{t}\lambda_j(u)du<\infty$ for some rate function $\lambda:[0,\infty)\to[0,\infty)$, $\{Y_\alpha(t)\}_{t\ge 0}$ is an independent inverse stable subordinator and $\mathcal{F}_t^j=\sigma\left(\mathcal{N}_j^\alpha(s),s\le t\right)\vee\sigma\left(Y_\alpha(s),s\ge 0\right)$. Also for any $T>0$, $\{X_j(\tau),\,\tau\,\, \text{is stopping time such that}\,\,Y_{\alpha}(\tau)\le T\}$ is uniformly integrable.
\end{theorem}
\begin{proof}
    Let $\{\mathcal{N}_1(t)\}_{t\ge 0}$, $\{\mathcal{N}_2(t)\}_{t\ge 0}$,~$\dots,\{\mathcal{N}_k(t)\}_{t\ge 0}$ be $k$ independent NPPs with respective rate functions $\lambda_1(t)$, $\lambda_2(t)$, $\dots,~\lambda_k(t)$. Let $\{Y_\alpha(t)\}_{t\ge 0}$ be an inverse stable subordinator independent of these NPPs. Then, by Theorem \ref{ngcp_npp_equality}, we have $\mathcal{M}(t)=\sum_{j=1}^{k}j\mathcal{N}_j(t)$ which implies $\mathcal{M}(Y_\alpha(t))=\sum_{j=1}^{k}j\mathcal{N}_j(Y_\alpha(t))$, that is, $\mathcal{M}^\alpha(t)=\sum_{j=1}^{k}j\mathcal{N}_j^\alpha(t)$ where $\{\mathcal{N}_j^\alpha(t)\}_{t\ge 0}$, $1\le j\le k$ are NTFPPs. Theorem \ref{Watanabe_characterization_NFPP} implies that $\{X_j(t)=\mathcal{N}_j^\alpha(t)-\Lambda_j(Y_\alpha(t))\}_{\t\ge 0}$ is a $\{\mathcal{F}_t^j\}_{t\ge 0}$-martingale and that for any $T>0$, the family $\{X_j(\tau),\,\tau\,\,$ is stopping time such\\ that$\,\,Y_{\alpha}(\tau)\le T\}$ is uniformly integrable. 
    
    Conversely, suppose that $\mathcal{M}^\alpha(t)=\sum_{j=1}^{k}j\mathcal{N}_j^\alpha(t)$, where $\{\mathcal{N}_j^\alpha(t)\}_{t\ge 0}$ are $k$ independent simple locally finite point processes such that $\{X_j(t)=\mathcal{N}_j^\alpha(t)-\Lambda_j(Y_\alpha(t))\}_{t\ge 0}$, $t\ge 0$ are right continuous $\{\mathcal{F}_t^j\}_{t\ge 0}$-martingales. Furthermore, suppose that for any $T>0$, the family $\{X_j(\tau),\,\tau\,\, \text{is stopping time such that}\,\,Y_{\alpha}(\tau)\le T\}$ is uniformly integrable. Then, by Theorem \ref{Watanabe_characterization_NFPP}, it follows that $\{\mathcal{N}_j^\alpha(t)\}_{t\ge 0}$ is a NTFPP. Consequently, using Theorem \ref{ngcp_npp_equality} along with definition \eqref{definition_NGFCP} of a NGFCP, we conclude that $\{\mathcal{M}^\alpha(t)\}_{t\ge 0}$ is a NGFCP.
\end{proof}
\begin{remark}
    For $\alpha=1$, Theorem \ref{Watanabe_characterization_NGFCP} reduces to the martingale characterization of a NGCP given by Theorem \ref{martingale_characterization_compensated_ngcp}.
\end{remark}
Since the weighted sum of martingales is a martingale for a suitable choice of filtration, the martingale property of a NGFCP shown below follows from Theorem \ref{Watanabe_characterization_NGFCP}.
\begin{corollary}\label{martingale_corollary_NGFCP}
    Let $\{\mathcal{M}^\alpha(t)\}_{t\ge 0}$ be a NGFCP. Then, the process $\{\mathcal{M}^\alpha(t)-\sum_{j=1}^{k}j\Lambda_j\left(Y_{\alpha}(t)\right)\}_{t\ge 0}$ is a $\{\mathcal{F}_t\}_{t\ge 0}$-martingale, where $\Lambda_j(t)=\int_{0}^{t}\lambda_j(u)du<\infty$ for some rate function $\lambda:[0,\infty)\to[0,\infty)$, $\{Y_\alpha(t)\}_{t\ge 0}$ is an inverse stable subordinator and $\mathcal{F}_t=\sigma\left(\mathcal{N}_1^\alpha(s),~\mathcal{N}_2^\alpha(s),\dots,~\mathcal{N}_k^\alpha(s),~s\le t\right)\vee\sigma\left(Y_\alpha(s),s\ge 0\right)$.
\end{corollary}
\begin{remark}
Taking $\alpha=1$ in Corollary \ref{martingale_corollary_NGFCP} we observe that for a NGCP $\{\mathcal{M}(t)\}_{t\ge 0}$, the process $\{\mathcal{M}(t)-\sum_{j=1}^{k}j\Lambda_j(t)\}_{t\ge 0}$ is a $\{\mathcal{F}_t\}_{t\ge 0}$-martingale where $\mathcal{F}_t = \sigma(\mathcal{M}(s),s\le t$\} (see \citet{Kataria2025}).
\end{remark}
The next result gives the martingale characterization of a non-homogeneous tempered generalized space time fractional counting process (NTGSTFCP), the generalized version of NTSTFPP, defined as
\begin{equation}\label{definition_NTGSTFCP}
    \mathcal{M}^{\alpha}_{\beta,\theta}(t)=\mathcal{M}\left(D_{\beta,\theta}(Y_{\alpha}(t))\right),
\end{equation}
where $\{D_{\beta,\theta}(t)\}_{t\ge 0}$ is a tempered stable subordinator and $\{Y_{\alpha}(t)\}_{t\ge 0}$ is an inverse stable subordinator independent of each other and also independent of the NGCP $\{\mathcal{M}(t)\}_{t\ge 0}$.
\begin{theorem}\label{Watanabe_characterization_NTGSTFCP}
    The process $\{\mathcal{M}^\alpha_{\beta,\theta}(t)\}_{t\ge 0}$ is a NTGSTFCP if and only if there exist $k$ simple locally finite point processes $\{\mathcal{N}^\alpha_{j\beta,\theta}(t)\}_{t\ge 0}$, $1\le j\le k$ such that $\mathcal{M}^\alpha_{\beta,\theta}(t)=\sum_{j=1}^{k}j\mathcal{N}^{\alpha}_{j\beta,\theta}(t)$ and $X_j(t)=\mathcal{N}^{\alpha}_{j\beta,\theta}(t)-\Lambda_j(D_{\beta,\theta}(Y_{\alpha}((t)))$, $t\ge 0$ is a right continuous $\{\mathcal{F}_t^j\}_{t\ge 0}$-martingale, where $\Lambda_j(t)=\int_{0}^{t}\lambda_j(u)du<\infty$ for some rate function $\lambda_j:[0,\infty)\to[0,\infty)$, $\{D_{\beta,\theta}(t)\}_{t\ge 0}$ is a tempered stable subordinator, $\{Y_{\alpha}(t)\}_{t\ge 0}$ is an independent inverse stable subordinator and $\mathcal{F}_t^j=\sigma\left(\mathcal{N}^{\alpha}_{j\beta,\theta}(s),s\le t\right)\vee\sigma\left(D_{\beta,\theta}(Y_{\alpha}((s)),s\ge 0\right)$. Moreover, for any $T>0$, $\{X_j(\tau),\tau\,\, \text{is stopping time such that}\,\,D_{\beta,\theta}(Y_{\alpha}((\tau))\le T\}$ is uniformly integrable.
\end{theorem}
\begin{proof}
The proof is similar to that of Theorem \ref{Watanabe_characterization_NGFCP} with $Y_{\alpha}(t)$ replaced by $D_{\beta,\theta}(Y_{\alpha}(t))$, where $\{D_{\beta,\theta}(Y_{\alpha}(t))\}_{t\ge 0}$ is independent of $\{\mathcal{M}(t)\}_{t\ge 0}$.
\end{proof}
 A non-homogeneous tempered generalized space fractional counting process (NTGSFCP), the generalized version of NTSFPP, is defined as 
\begin{equation}\label{definition_NTGSFCP}
    \mathcal{M}_{\beta,\theta}(t)=\mathcal{M}\left(D_{\beta,\theta}(t)\right),
\end{equation}
where $\{\mathcal{M}(t)\}_{t\ge 0}$ is a NGCP and $\{D_{\beta,\theta}(t)\}_{t\ge 0}$ is an independent tempered stable subordinator. The martingale characterization of NTGSFCP follows from Theorem \ref{Watanabe_characterization_NTGSTFCP} by taking $\alpha=1$ and noting that $Y_1(t)=t$ a.s.
Next, we investigate the martingale characterization of the mixed fractional extension of a NGCP. The non-homogeneous mixed fractional Poisson process (NMFPP) is defined as
\begin{equation*}
    \mathcal{N}^{\alpha_1,\alpha_2}(t)=\mathcal{N}\left(Y_{\alpha_1,\alpha_2}(t)\right),
\end{equation*}
where $\{\mathcal{N}(t)\}_{t\ge 0}$ is a NPP and $\{Y_{\alpha_1,\alpha_2}(t)\}_{t \ge 0}$ is an independent inverse mixed stable subordinator. The corresponding generalized process, referred to as the non-homogeneous mixed fractional counting process (NMFCP), is defined as
\begin{equation*}\label{definition_NMFCP}
    \mathcal{M}^{\alpha_1,\alpha_2}(t)=\mathcal{M}\left(Y_{\alpha_1,\alpha_2}(t)\right),
\end{equation*}
where $\{\mathcal{M}(t)\}_{t\ge 0}$ is a NGCP and $\{Y_{\alpha_1,\alpha_2}(t)\}_{t\ge 0}$ is an independent inverse mixed stable subordinator. The following theorem presents the martingale characterization of a NMFCP.
\begin{theorem}\label{martingale_characterization_NMFCP}
    The process $\{\mathcal{M}^{\alpha_1,\alpha_2}(t)\}_{t\ge0}$ is a NMFCP if and only if there exist $k$ simple locally finite point processes $\{\mathcal{N}_j^{\alpha_1,\alpha_2}(t)\}_{t\ge0}$, $1\le j\le k$ such that $\mathcal{M}^{\alpha_1,\alpha_2}(t)=\sum_{j=1}^{k}j\,\mathcal{N}_j^{\alpha_1,\alpha_2}(t)$ and $X_j(t):=\mathcal{N}_j^{\alpha_1,\alpha_2}(t)-\Lambda_j(Y_{\alpha_1,\alpha_2}(t))$, $t\ge0$, is a right-continuous $\{\mathcal{F}_t^j\}_{t\ge 0}$-martingale, where $\Lambda_j(t)=\int_0^t \lambda_j(u)\,du<\infty$ for some rate function $\lambda_j:[0,\infty)\to[0,\infty)$, $\{Y_{\alpha_1,\alpha_2}\}_{t\ge 0}$ is an mixed inverse stable subordinator and  $\mathcal{F}_t^j=\sigma(\mathcal{N}_j^{\alpha_1,\alpha_2}(s):s\le t)\vee\sigma(Y_{\alpha_1,\alpha_2}(s):s\le t)$. Moreover, for any $T>0$, $\{X_j(\tau): \tau \text{ is a stopping time with } Y_{\alpha_1,\alpha_2}(\tau)\le T\}$ is uniformly integrable.
\end{theorem}
\begin{proof}
    The proof is similar to that of Theorem \ref{Watanabe_characterization_NGFCP} with $Y_{\alpha}(t)$ replaced by $Y_{\alpha_1,\alpha_2}(t)$ and hence omitted.
\end{proof}
Finally, we obtain the martingale characterizations of Skellam processes. In general, a Skellam process is defined as the difference of two idependent counting processes. The non-homogeneous generalized Skellam process (NGSP) is defined as the difference of two independent NGCPs. Specifically, let $\{\mathcal{M}_{1}(t)\}_{t \ge 0}$ and $\{\mathcal{M}_{2}(t)\}_{t \ge 0}$ be two independent NGCPs with respective intensity functions $\lambda_{j} : [0,\infty) \rightarrow [0,\infty)$ and $\mu_{j} : [0,\infty) \rightarrow [0,\infty)$ for $j = 1, 2, \ldots, k$, and with initial conditions $\mathcal{M}_{1}(0) = \mathcal{M}_{2}(0) = 0$. Then the NGSP is given by $\{\mathcalboondox{S}(t) = \mathcal{M}_{1}(t) - \mathcal{M}_{2}(t)\}_{t\ge 0}$ whose martingale characterization is given as follows.

\begin{theorem}\label{Watanabe_characterization_NGSP}
Let $\{\mathcal{S}_j(t)\}_{t\ge 0}$ be $k$ independent Skellam processes with parameters $\lambda_j(t)$ and $\mu_j(t)$ for $j=1,2,\dots,k$ such that $\mathcal{S}_j(0)=0$. Then $\{\sum_{j=1}^{k}j\mathcal{S}_j(t)-\sum_{j=1}^{k}j(\Lambda_j(t)-T_j(t)\}_{t\ge 0}$, where $\Lambda_j(t)=\int_{o}^{t}\lambda_j(u)du$ and $T_j(t)=\int_{o}^{t}\mu_j(u)du$, is a $\mathcal{F}_t$-martingale for $\mathcal{F}_t=\sigma(\mathcal{S}_1(t),\dots,\mathcal{S}_k(t))$ if and only if the weighted process $\{\sum_{j=1}^{k}j\mathcal{S}_j(t)\}_{t\ge 0}$ is a NGSP.
\end{theorem}
\begin{proof}
Let $\{\mathcal{N}_{1j}(t)\}_{t\ge 0}$ and $\{\mathcal{N}_{2j}(t)\}_{t\ge 0}$, $j=1,\dots,k$, be independent NPPs with rates $\lambda_j(t)$ and $\mu_j(t)$ respectively, and define the non-homogeneous Skellam process (NSP) by $\mathcal{S}_j(t)=\mathcal{N}_{1j}(t)-\mathcal{N}_{2j}(t)$. Since a weighted sum of independent NPPs is a NGCP (see Theorem \ref{ngcp_npp_equality}), the process $\{\sum_{j=1}^{k}j\mathcal{S}_j(t)\}_{t\ge 0}$ is a NGSP. By Theorem~\ref{martingale_characterization_compensated_ngcp}, the compensated processes $\{\sum_{j=1}^{k}j\mathcal{N}_{1j}(t)-\sum_{j=1}^{k}j\Lambda_j(t)\}_{t\ge 0}$ and $\{\sum_{j=1}^{k}j\mathcal{N}_{2j}(t)-\sum_{j=1}^{k}jT_j(t)\}_{t\ge 0}$ are $\mathcal{F}_t$-martingales, and thus their difference $\{\sum_{j=1}^k j\mathcal{S}_j(t)-\sum_{j=1}^k j(\Lambda_j(t)-T_j(t))\}_{t\ge 0}$ is a $\mathcal{F}_t$-martingale. 

Conversely, if $\{\sum_{j=1}^{k}j\mathcal{S}_j(t)-\sum_{j=1}^{k}j(\Lambda_j(t)-T_j(t)\}_{t\ge 0}$ is a $\mathcal{F}_t$-martingale, then by the equivalence of compensated and exponential martingale forms, for every $u\in\mathbb{R}$, the process
$\left\{\exp X(t)\right\}_{t\ge 0}$ is also a $\mathcal{F}_t$-martingale, where $X(t)=u\sum_{j=1}^{k}j\mathcal{S}_j(t)
-\sum_{j=1}^{k}\left[(e^{uj}-1)\Lambda_j(t)+(e^{-uj}-1)T_j(t)\right]$. 
By definition of a martingale, we have
\begin{align}
    &\mathbb{E}\!\left(
    e^{u\sum_{j=1}^{k}j\mathcal{S}_j(t)
    -\sum_{j=1}^{k}\!\left[(e^{uj}-1)\Lambda_j(t)+(e^{-uj}-1)T_j(t)\right]}
    \Big|\mathcal{F}_s
    \right)=
    e^{u\sum_{j=1}^{k}j\mathcal{S}_j(s)
    -\sum_{j=1}^{k}\!\left[(e^{uj}-1)\Lambda_j(s)+(e^{-uj}-1)T_j(s)\right]}\nonumber\\
    \implies&\mathbb{E}\!\left(e^{u\sum_{j=1}^{k}j(\mathcal{S}_j(t)-\mathcal{S}_j(s))}\Big|\mathcal{F}_s\right)=\exp\!\left\{
    \sum_{j=1}^{k}\!\left[(e^{uj}-1)\Lambda_j(s,t)+(e^{-uj}-1)T_j(s,t)\right]
    \right\}\nonumber\\
    \implies&\mathbb{E}\!\left(e^{u\sum_{j=1}^{k}j(\mathcal{S}_j(t)-\mathcal{S}_j(s))}\right)=\exp\!\left\{
    \sum_{j=1}^{k}\!\left[(e^{uj}-1)\Lambda_j(s,t)+(e^{-uj}-1)T_j(s,t)\right]
    \right\}\label{ngsp_conditional_eq}
\end{align}
which is the m.g.f. of a NGSP with parameters $\Lambda_j(s,t)=\Lambda_j(t)-\Lambda_j(s)=\int_{s}^{t}\lambda_j(u)du$ and 
$T_j(s,t)=T_j(t)-T_j(s)=\int_{s}^{t}\mu_j(u)du$. Putting $s=0$ in \eqref{ngsp_conditional_eq}, we conclude that $\{\sum_{j=1}^{k}j\mathcal{S}_j(t)\}_{t\ge 0}$ is equal in distribution to a NGSP with parameters $\Lambda_j(t)$ and $T_j(t)$. 
For $i\in\mathbb{N}$, consider the increment 
\begin{equation*}
\Delta_i\left(\sum_{j=1}^{k}j\mathcal{S}_j(t)\right)
=\sum_{j=1}^{k}j\mathcal{S}_j(t_i)-\sum_{j=1}^{k}j\mathcal{S}_j(t_{i-1})
=\sum_{j=1}^{k}j\left(\mathcal{N}_{1j}(t_i)-\mathcal{N}_{1j}(t_{i-1})\right)
-\sum_{j=1}^{k}j\left(\mathcal{N}_{2j}(t_i)-\mathcal{N}_{2j}(t_{i-1})\right). 
\end{equation*}
Since processes $\{\mathcal{N}_{1j}(t)\}_{t\ge 0}$ are independent, the random variables $\{\Delta_i(\mathcal{N}_{1j}(t))=\mathcal{N}_{1j}(t_i)-\mathcal{N}_{1j}(t_{i-1})\}$ are independent for fixed $j$. Similarly, the random variables $\{\Delta_i(\mathcal{N}_{2j}(t))=\mathcal{N}_{2j}(t_i)-\mathcal{N}_{2j}(t_{i-1})\}$ are independent. Since $\{\mathcal{N}_{1j}(t)\}_{t\ge 0}$ and $\{\mathcal{N}_{2j}(t)\}_{t\ge 0}$ are mutually independent, the collection $\{\Delta_i(\mathcal{N}_{1j}(t)),\Delta_i(\mathcal{N}_{2j}(t)):1\le j\le k\}$ is independent. Moreover $\Delta_i\left(\sum_{j=1}^{k}j\mathcal{S}_j(t)\right)$ being a linear combination of these independent random variables implies the collection $\{\Delta_i(\sum_{j=1}^{k}j\mathcal{S}_j(t)):i\in\mathbb{N}\}$ is independent, that is, the process $\{\sum_{j=1}^{k}j\mathcal{S}_j(t)\}_{t\ge0}$ has independent increments.
 Hence, $\{\sum_{j=1}^{k}j\mathcal{S}_j(t)\}_{t\ge0}$ is a NGSP.
\end{proof}

The following result provides the Watanabe characterization of a non-homogeneous time fractional Skellam process (NFSP) which may be defined as the difference of two independent NTFPPs.  
\begin{theorem}\label{Watanabe_characterization_NFSP}
    Let $\{\mathcal{K}_1(t)\}_{t\ge 0}$ be a simple locally finite point process. Then, $\{\mathcal{K}_1(t)\}_{t\ge 0}$ is a NFSP if and only if there exist rate functions $\lambda:[0,\infty)\to[0,\infty)$ and $\mu:[0,\infty)\to[0,\infty)$ with $\Lambda(t)=\int_{0}^{t}\lambda(u)du<\infty$ and $T(t)=\int_{o}^{t}\mu(u)du<\infty$ and two independent inverse stable subordinators $\{Y_{\alpha_1}(t)\}_{t\ge 0}$ and $\{Y_{\alpha_2}(t)\}_{t\ge 0}$ such that the process 
    \begin{equation}\label{martingale_form_nfsp}
    \{\mathcal{X}(t)\}_{t\ge 0}=\{\mathcal{K}_1(t)-\left(\Lambda(Y_{\alpha_1}(t))-T(Y_{\alpha_2}(t))\right)\}_{t\ge 0}    
    \end{equation} 
    is a right-continuous martingale with respect to the induced filtration $\mathcal{F}_t=\sigma\left(\mathcal{K}_1(s),s\le t\right)\vee\sigma\left(Y_{\alpha_1}(s), s\ge 0\right)\vee\sigma\left(Y_{\alpha_2}(s), s\ge 0\right)$ and for any $T>0$,
    \begin{equation}\label{stopping_times_martingale_nfsp}
        \{\mathcal{X}(\tau),\,\tau\,\, \text{is stopping time such that}\,\,\max\left(Y_{\alpha_1}(\tau),Y_{\alpha_2}(\tau)\right)\le T\}
    \end{equation}
    is uniformly integrable.
\end{theorem}
\begin{proof}
Assume $\{\mathcal{K}_1(t)\}_{t\ge 0}$ is a NFSP. By definition,  $\mathcal{K}_1(t)=\mathcal{N}_1(Y_{\alpha_1}(t))-\mathcal{N}_2(Y_{\alpha_2}(t))$, where $\{\mathcal{N}_1(Y_{\alpha_1}(t))\}_{t\ge0}$ and $\{\mathcal{N}_2(Y_{\alpha_2}(t))\}_{t\ge0}$ are independent NTFPPs with parameters $\Lambda(t)$ and $T(t)$, and $\{Y_{\alpha_1}(t)\}_{t\ge 0}$ and $\{Y_{\alpha_2}(t)\}_{t\ge 0}$ are two independent inverse stable subordinators that are also independent of these processes. 

By Theorem \ref{Watanabe_characterization_NFPP}, $\{\mathcal{N}_1(Y_{\alpha_1}(t))-\Lambda(Y_{\alpha_1}(t))\}_{t\ge 0}$ and $\{\mathcal{N}_2(Y_{\alpha_2}(t))-T(Y_{\alpha_2}(t))\}_{t\ge 0}$
 are $\mathcal{F}_t$-martingales. This implies that $\{\mathcal{X}(t)\}_{t\ge 0}$ is a $\mathcal{F}_t$-martingale since the difference of two martingales with respect to a common filtration is a martingale. Now consider the processes $\{\mathcal{N}_1(Y_{\alpha_1}(\tau))-\Lambda(Y_{\alpha_1}(\tau))\}$ and $\{\mathcal{N}_2(Y_{\alpha_2}(\tau))-T(Y_{\alpha_2}(\tau))\}$ where $\tau$ is a stopping time such that $\max\left(Y_{\alpha_1}(\tau),~Y_{\alpha_2}(\tau)\right)\le T$. By Theorem \ref{Watanabe_characterization_NFPP}, these processes are uniformly integrable. Since uniform integrability is closed under linear combinations, the process given by \eqref{stopping_times_martingale_nfsp} is also uniformly integrable.

 The converse part follows from the converse part in the proof of Theorem \ref{Watanabe_characterization_NFPP} applied to the two NTFPPs $\{\mathcal{N}_1(Y_{\alpha_1}(t))\}_{t\ge 0}$ and $\{\mathcal{N}_2(Y_{\alpha_2}(t))\}_{t\ge 0}$, and then taking their difference.
\end{proof}
If a NFSP is defined by time-changing a non-homogeneous Skellam process (NSP) with an independent inverse stable subordinator, the following result provides its Watanabe characterization.
\begin{theorem}\label{Watanabe_characterization_NFSP-II}
    Let $\{\mathcal{K}_2(t)\}_{t\ge 0}$ be a simple locally finite point process. Then, $\{\mathcal{K}_2(t)\}_{t\ge 0}$ is a NFSP if and only if there exist rate functions $\lambda:[0,\infty)\to[0,\infty)$ and $\mu:[0,\infty)\to[0,\infty)$ with $\Lambda(t)=\int_{0}^{t}\lambda(u)du<\infty$ and $T(t)=\int_{o}^{t}\mu(u)du<\infty$ and an independent inverse stable subordinator $\{Y_{\alpha}(t)\}_{t\ge 0}$ such that the process 
    \begin{equation}\label{martingale_form_nfsp-II}
    \{\mathcal{X}(t)\}_{t\ge 0}=\{\mathcal{K}_2(t)-\left(\Lambda(Y_{\alpha}(t))-T(Y_{\alpha}(t))\right)\}_{t\ge 0}    
    \end{equation} 
    is a right-continuous martingale with respect to the induced filtration $\mathcal{F}_t=\sigma\left(\mathcal{K}_2(s),s\le t\right)\vee\sigma\left(Y_{\alpha}(s), s\ge 0\right)$ and for any $T>0$,
    \begin{equation}\label{stopping_times_martingale_nfsp-II}
        \{\mathcal{X}(\tau),\,\tau\,\, \text{is stopping time such that}\,\,Y_{\alpha}(\tau)\le T\}
    \end{equation}
    is uniformly integrable.
\end{theorem}
\begin{proof}
For a non-homogeneous Skellam process $\{\mathcal{S}(t)\}_{t\ge 0}$, note that $\mathcal{K}_2(t)\big|Y_\alpha(t)=\mathcal{S}\left(Y_\alpha(t)\right)\big|Y_\alpha(t)$ has a non-homogeneous Skellam distribution with mean $\Lambda(Y_\alpha(t)) - T(Y_\alpha(t))$. It can also be shown that $\mathbb{E}(X(\tau))=0$ and $\mathbb{E}\left(X^2(\tau)\right)\le \Lambda(T) + T(T)$. The proof of the result follows in a manner analogous to the proof of Theorem~\ref{Watanabe_characterization_NFPP}.    
\end{proof}

The following result establishes the Watanabe characterization of NGFSP obtained by time-changing a NGSP with an independent inverse stable subordinator.
\begin{proposition}\label{Watanabe_characterization_NGFSP-II}
The process $\{\mathcalboondox{S}^{\alpha}(t)\}_{t\ge 0}$ is a NGFSP if and only if there exist $k$ simple locally finite point processes $\{\mathcal{S}^{\alpha}_{j}(t)\}_{t\ge 0}$, $1\le j\le k$ such that $\mathcalboondox{S}^{\alpha}(t)=\sum_{j=1}^{k}j\mathcal{S}^{\alpha}_{j}(t)$ and $X_j(t)=\mathcal{S}^{\alpha}_{j}(t)-(\Lambda_j(Y_{\alpha}(t))-T_j(Y_{\alpha}(t))$, $t\ge 0$ is a right-continuous $\{\mathcal{F}_t^j\}_{t\ge 0}$-martingale, where $\mathcal{F}_t^j=\sigma\left(\mathcal{S}_j^{\alpha}(s),s\le t\right)\vee\sigma\left(Y_{\alpha}(s),s\ge 0\right)$ and $\{Y_\alpha(t)\}_{t\ge 0}$ is an inverse stable subordinator independent of $\{\mathcal{S}_j(t)\}_{t\ge 0}$. Moreover, for any $T>0$, $\{X_j(\tau),\tau\,\, \text{is stopping time such that}\,\,(Y_{\alpha}(\tau))\le T\}$ is uniformly integrable.
\end{proposition}
\begin{proof}
We know that a NGSP can be expressed as the difference between two independent NGCPs, each of which admits a unique representation as a weighted sum of $k$ independent NPPs by Theorem~\ref{ngcp_npp_equality}. Specifically, the NGSP can be written as $\mathcalboondox{S}(t)=\sum_{j=1}^{k} j\mathcal{S}_j(t)$, where $\mathcal{S}_j(t)=\mathcal{N}_{1j}(t)-\mathcal{N}_{2j}(t)$, with $\{\mathcal{N}_{1j}(t)\}_{t\ge0}$ and $\{\mathcal{N}_{2j}(t)\}_{t\ge0}$ denoting independent NPPs having intensity functions $\lambda_j(t)$ and $\mu_j(t)$ respectively. Then the NGFSP $\{\mathcalboondox{S}^{\alpha}(t)\}_{t\ge 0}$ is given by $\mathcalboondox{S}^{\alpha}(t)=\mathcalboondox{S}(Y_{\alpha}(t))=\sum_{j=1}^{k} j\mathcal{S}_j(Y_{\alpha}(t))=\sum_{j=1}^{k} j\mathcal{S}^{\alpha}_j(t)$, where $\{\mathcal{S}_j^{\alpha}(t)\}_{t\ge 0}$ is a NFSP. Therefore for each $\{\mathcal{S}_j^{\alpha}(t)\}_{t\ge 0}$, it follows from Theorem~\ref{Watanabe_characterization_NFSP-II} that the process $\{X_j(t)=\mathcal{S}^{\alpha}_{j}(t)-(\Lambda_j(Y_{\alpha}(t)) - T_j(Y_{\alpha}(t)))\}_{t\ge 0}$ is a right-continuous martingale with respect to the filtration $\mathcal{F}_t^j=\sigma(\mathcal{S}^{\alpha}_{j}(s),s\le t)\vee\sigma(Y_{\alpha}(s),s\ge0)$. The uniform integrability of $\{X_j(\tau)\}$, for stopping times $\tau$ satisfying $Y_{\alpha}(\tau)\le T$, follows from the corresponding property established in Theorem~\ref{Watanabe_characterization_NFSP-II}.

Conversely, suppose there exist simple locally finite point processes $\{\mathcal{S}^{\alpha}_{j}(t)\}_{t\ge0}$ and an independent inverse stable subordinator $\{Y_{\alpha}(t)\}_{t\ge 0}$ such that each $\{X_j(t)=\mathcal{S}^{\alpha}_{j}(t)-(\Lambda_j(Y_{\alpha}(t))-T_j(Y_{\alpha}(t)))\}_{t\ge 0}$ is a right-continuous $\mathcal{F}_t^j$-martingale. Then by Theorem~\ref{Watanabe_characterization_NFSP-II}, each component process $\{\mathcal{S}^{\alpha}_{j}(t)\}_{t\ge 0}$ is a NFSP. Consequently, their weighted sum $\{\mathcalboondox{S}^{\alpha}(t)=\sum_{j=1}^{k} j\,\mathcal{S}^{\alpha}_{j}(t)\}_{t\ge 0}$ is a NGFSP. This completes the proof.
\end{proof}
If the NGFSP is obtained by taking the difference of two independent NGFCPs, the following result provides its martingale characterization.

\begin{proposition}\label{Watanabe_characterization_NGFSP}
The process $\{\mathcalboondox{S}^{\alpha_1,\alpha_2}(t)\}_{t\ge 0}$ is a NGFSP if and only if there exist $k$ simple locally finite point processes $\{\mathcal{S}^{\alpha_1,\alpha_2}_{j}(t)\}_{t\ge 0}$, $1\le j\le k$ such that $\mathcalboondox{S}^{\alpha_1,\alpha_2}(t)=\sum_{j=1}^{k}j\mathcal{S}^{\alpha_1,\alpha_2}_{j}(t)$ and $X_j(t)=\mathcal{S}^{\alpha_1,\alpha_2}_{j}(t)-(\Lambda_j(Y_{\alpha_1}(t))-T_j(Y_{\alpha_2}(t))$, $t\ge 0$ is a right continuous $\{\mathcal{F}_t^j\}_{t\ge 0}$-martingale, where $\mathcal{F}_t^j=\sigma\left(\mathcal{S}_j^{\alpha_1,\alpha_2}(s),s\le t\right)\vee\sigma\left(Y_{\alpha_1}(s),s\ge 0\right)\vee\sigma\left(Y_{\alpha_2}(s),s\ge 0\right)$, and $\{Y_{\alpha_1}(t)\}_{t\ge 0}$ and $\{Y_{\alpha_2}(t)\}_{t\ge 0}$ are two independent inverse stable subordinators independent of $\{S_j(t)\}_{t\ge 0}$. Moreover, for any $T>0$, $\{X_j(\tau),\tau\,\, \text{is stopping time such that}\\\,\,\max\left(Y_{\alpha_1}(\tau),Y_{\alpha_2}(\tau)\right)\le T\}$ is uniformly integrable.
\end{proposition}
\begin{proof}
Let $\{\mathcal{M}_{1}^{\alpha_1}(t)\}_{t\ge 0}$ and $\{\mathcal{M}_{2}^{\alpha_2}(t)\}_{t\ge 0}$ be two independent NGFCPs. By definition, we have $\mathcalboondox{S}^{\alpha_1,\alpha_2}(t)=\mathcal{M}_{1}(Y_{\alpha_1}(t))-\mathcal{M}_{2}(Y_{\alpha_2}(t))=\sum_{j=1}^{k}j\left(\mathcal{N}_{1j}(Y_{\alpha_1}(t))-\mathcal{N}_{2j}(Y_{\alpha_2}(t))\right)=\sum_{j=1}^{k}j(\mathcal{N}_{1j}^{\alpha_1}(t)-\mathcal{N}_{2j}^{\alpha_2}(t))=\sum_{j=1}^{k}j\mathcal{S}_{j}^{\alpha_1,\alpha_2}(t)$, where $\{\mathcal{S}_{j}^{\alpha_1,\alpha_2}(t)\}_{t\ge 0}$ is a NFSP. The result follows from Theorem~\ref{Watanabe_characterization_NFSP} together with the fact that a weighted sum of martingales (here $\{\mathcal{S}_{j}^{\alpha_1,\alpha_2}(t)\}_{t\ge 0}$) with respect to a common filtration (here $\{\mathcal{F}_t^j\}_{t\ge 0}$) is a martingale.
\end{proof}
Next, consider the non-homogeneous tempered generalized space-time fractional Skellam process (NTGSTFSP) and its special case NTGSFSP. If they are defined via time-change, then the martingale characterizations of NTGSTFSP and NTGSFSP follow from Proposition ~\ref{Watanabe_characterization_NGFSP-II} by replacing $Y_\alpha(t)$ with $D_{\beta,\theta}(Y_\alpha(t))$ and $D_{\beta,\theta}(t))$ respectively. However, if they are defined as the difference of two independent counting processes, then the martingale characterizations of NTGSTFSP and NTGSFSP follow from Proposition ~\ref{Watanabe_characterization_NGFSP} by replacing $Y_{\alpha_i}(t)$ with $D_{\beta_i,\theta_i}(Y_{\alpha_i}(t))$ and $D_{\beta_i,\theta_i}(t))$ for $i=1,2$ respectively. The non-homogeneous mixed fractional Skellam process (NMFSP) admits an analogous characterization when the process is defined via an inverse mixed stable subordinator.
Finally, the martingale properties of NGSP, NGFSP, NMFCP, NMFSP, NTGSFCP, NTGSTFCP, NTGSFSP and NTGSTFSP can be obtained from their respective martingale characterizations similar to Corollary \ref{martingale_corollary_NGFCP}.

\section*{Declarations}
\noindent 
{\bf Conflict of Interest}. The authors declare that they have no conflict of interest.

\noindent
{\bf Data availability}. No data was used for the research described in this article.

\bibliographystyle{plainnat}
\bibliography{NGFSP}

\end{document}